\documentclass{amsart}

\usepackage{amsmath,amssymb,amsthm}

\usepackage{graphicx}
\graphicspath{{./figure/}}

\usepackage{enumerate,bm,color,here,url,hyperref}

\makeatletter
\@namedef{subjclassname@2020}{\textup{2020} Mathematics Subject Classification}
\makeatother

\theoremstyle{plain}
	\newtheorem{thm}{Theorem}[section]
	\newtheorem{lem}[thm]{Lemma}
	\newtheorem{prop}[thm]{Proposition}
	\newtheorem{cor}[thm]{Corollary}

\theoremstyle{definition}
	\newtheorem{dfn}[thm]{Definition}
\theoremstyle{remark}
	\newtheorem{rem}[thm]{Remark}

\newcommand{\fig}[3][width=12cm]{
\begin{figure}[htbp]
	\centering 
	\includegraphics[#1,clip]{#2} 
	\caption{#3} 
	\label{fig:#2}
\end{figure}}

\newcommand{\twofig}[6]{
\begin{figure}[htbp]
\begin{tabular}{cc}
\begin{minipage}{0.5\hsize}
	\centering 
	\includegraphics[#1,clip]{#2} 
	\caption{#3} 
	\label{fig:#2}
\end{minipage}
\begin{minipage}{0.5\hsize}
	\centering 
	\includegraphics[#4,clip]{#5} 
	\caption{#6} 
	\label{fig:#5}
\end{minipage}
\end{tabular}
\end{figure}}

\DeclareMathOperator{\diag}{diag}
\DeclareMathOperator{\Ell}{Ell}
\DeclareMathOperator{\Span}{span}
\DeclareMathOperator{\tr}{tr}
\DeclareMathOperator{\vol}{vol}

\begin{document}

\title{A mathematical model of network elastoplasticity}
\date{}

\author{Hiroki KODAMA}
\address[H. Kodama]{WPI - Advanced Institute for Materials Research (WPI-AIMR), Tohoku University, 
2-1-1 Katahira, Aoba-ku, Sendai-shi, Miyagi 980-8577, Japan 
(RIKEN iTHEMS, 2-1 Hirosawa, Wako-shi, Saitama 351-0198, Japan)}
\email{kodamahiroki@gmail.com}

\author{Ken'ichi YOSHIDA}
\address[K. Yoshida]{Department of Mathematics, Saitama University, 255 Shimo-Okubo, Sakura-ku, Saitama-shi, Saitama 338-8570, Japan}
\email{kyoshida@mail.saitama-u.ac.jp}

\begin{abstract}
We introduce a mathematical model, based on networks, for the elasticity and plasticity of materials. 
We define the tension tensor for a periodic graph in a Euclidean space, 
and we show that the tension tensor expresses elasticity under deformation. 
Plasticity is induced by local moves on a graph. 
The graph is described in terms of the weights of edges, 
and we discuss how these weights affect the plasticity. 
\end{abstract}

\maketitle

\section{Introduction}
\label{section:intro}

The field of topological crystallography was initially introduced by Kotani and Sunada 
\cite{kotanisunada00albanese,kotanisunada00jacobian,kotanisunada01,kotanisunada03,sunada12} 
as a part of discrete geometric analysis. 
One of the main objects of their study is a \emph{net}, 
that is, a periodic graph realized in $\mathbb{R}^{N}$. 
The energy of a net is defined as an analogue of the Dirichlet energy of a Riemannian manifold. 
In other words, one can say that the energy of a net is the total potential energy of springs, viewing edges as linear springs with rest lengths equal to zero. 
Harmonic and standard nets are defined as energy-minimizing nets under certain conditions, 
and they are regarded as equilibrium states. 
Nets have been used as models of crystals.

In this paper, we suggest that 
the energy of a net induces a model of hyperelastic materials. 
Here, hyperelasticity is the property from which 
stress under deformation is derived using an energy density function. 
To describe the deformation of a net, 
we introduce the notion of a \emph{tension tensor}, 
which is regarded as multivariate energy. 
Further, a standard net is characterized by the tension tensor. 
We show that the Cauchy stress tensor is also expressed by the tension tensor. 
Furthermore, if the graph structure is preserved, 
the elasticity at the macro-scale is also determined by the tension tensor; 
otherwise, a departure from elasticity, 
known as plasticity, occurs. 
To describe the manner in which the graph structure changes, 
we consider two types of local moves: contraction and splitting. 
We define a condition for a local move 
and introduce two models of deformation concerning plasticity. 
This enables us to draw the stress--strain curve.

Our model is motivated by the structure of thermoplastic elastomers (TPEs). 
A TPE is a polymeric material with the rubber elasticity and is remoldable at high temperatures. 
A typical TPE consists of ABA triblock copolymers, 
in which monomers of types A and B are arranged in a sequence such as A$\cdots$AB$\cdots$BA$\cdots$A. 
ABA triblock copolymers of a certain type form two domains 
consisting of monomers A and B. 
This structure is called microphase separation. 
We consider a structure such that 
each component consisting of monomer A is a ball, 
as shown in Figure~\ref{fig:copolymer.pdf}. 
This is called a spherical structure. 
The domains consisting of monomers A and B are called hard and soft domains, respectively. 
The theoretical and numerical treatments of block copolymers are explained in the book by Fredrickson~\cite{fredrickson06}, to which we refer the reader for further information. 

\fig[width=10cm]{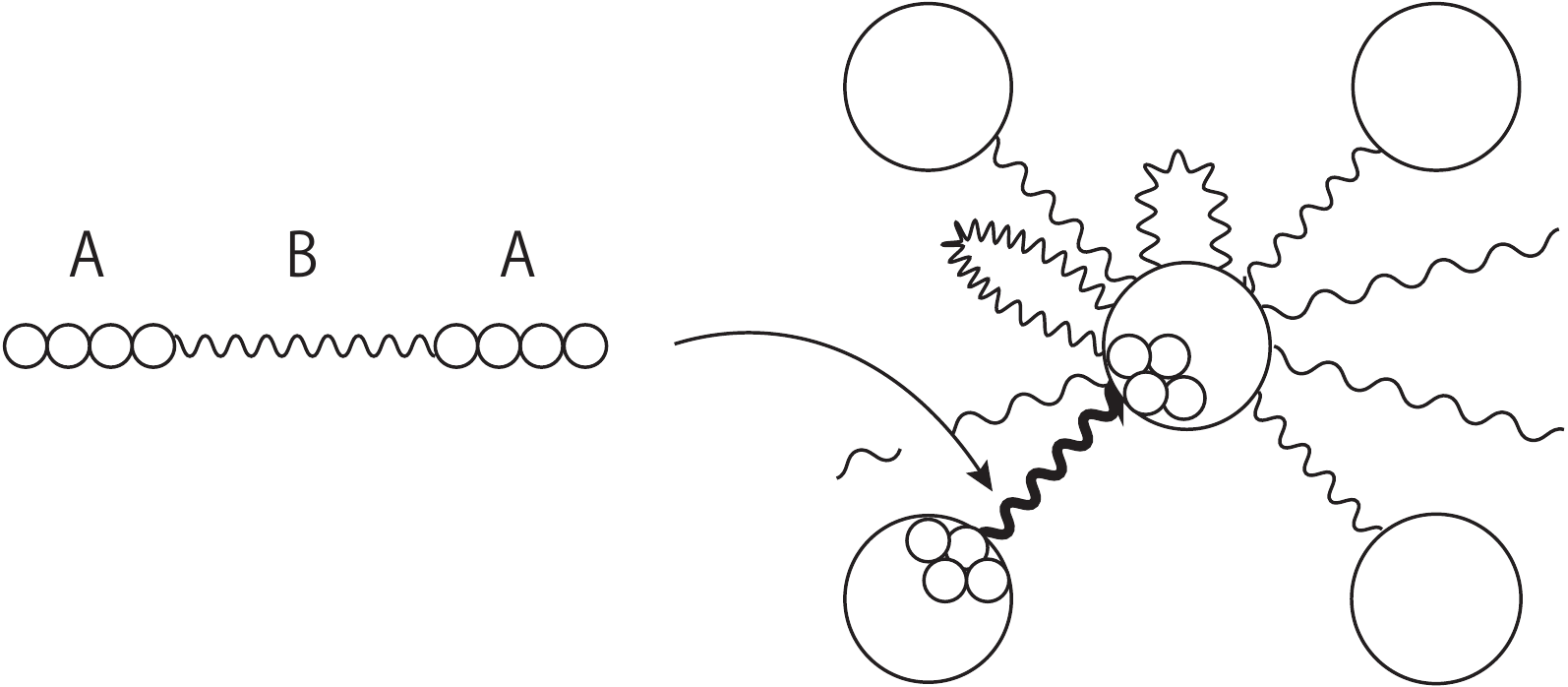}{A spherical structure formed from ABA triblock copolymers}

In our model, the hard and soft domains correspond to the vertices and edges, respectively. 
More precisely, a hard domain is a vertex, 
and a polymer chain in the soft domain is an edge. 
The endpoints of an edge are the (possibly single) hard domains 
that contain monomer A of the copolymer. 
The obtained graph may have loops and multiple edges.

The network structure of polymers induces rubber elasticity. 
The random motion of polymer chains in the soft domain gives rise to entropic forces. 
A hard domain functions as a cross-link. 
In our approximation, 
we ignore the maximal length of the chain and the interaction between chains. 
If a chain moves randomly, 
the tension on the chain is proportional to the distance between the endpoints. 
This setting is consistent with the definition of the energy of a net. 
Suppose that the polymers can move freely while preserving the network structure. 
Additionally, we obtain a harmonic net in equilibrium. 
Since harmonicity is preserved by affine deformation, 
the affine assumption in the classical theory of rubber elasticity holds. 
Additional details of rubber elasticity are provided in the book by Treloar~\cite{treloar75}.

\fig[width=12cm]{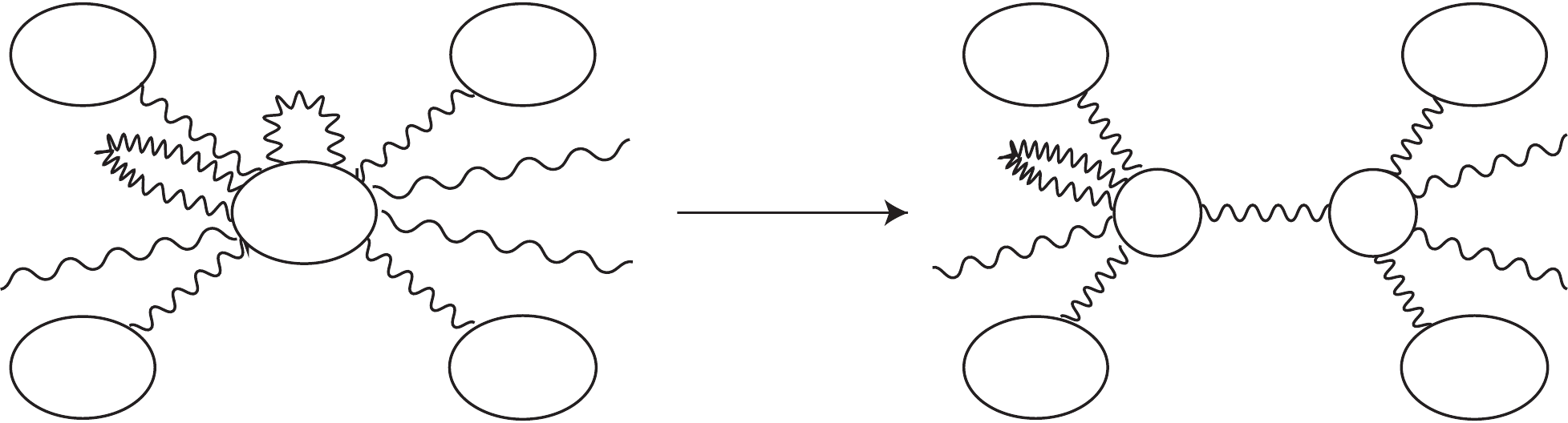}{Splitting of a hard domain. In this example, one loop becomes a non-loop edge between the new domains.}

The hard domains of a TPE are less robust than 
the cross-links of vulcanized rubber 
because each hard domain is aggregated by intermolecular forces. 
The network structure of a TPE may change under deformation, 
as observed in simulation \cite{aoyagihondadoi02,moritamiyamotokotani20} 
and by conducting experiments \cite{liuliangnakajima20}. 
For example, a hard domain may split 
as shown in Figure~\ref{fig:polymer-split.pdf}. 
Further, a loop may become a non-loop edge between the new domains. 
Conversely, two hard domains may contract. 
These moves cause plasticity. 
Although other moves may occur, 
we consider only contractions and splittings.

%Structure of the paper 

In Section~\ref{section:def}, we give preliminary definitions of nets. 
The graphs we use are weighted, and their weights can be regarded as the number of edges. 
Different types of polymers may contribute different weights. 

In Section~\ref{section:tension}, we introduce the tension tensor. 
This is visualized by an ellipsoid.

In Section~\ref{section:stress}, we consider the elasticity of nets under deformation. 
Based on a physical argument, the Cauchy stress tensor is derived from the tension tensor. 
We also consider the stress under uniaxial extension. 
The Young's modulus of a standard net is determined by the energy per unit volume. 

In Section~\ref{section:move}, we define local moves. 
The contraction of two vertices is a natural operation. 
A splitting of a vertex is an inverse operation of contraction. 
The sum of weights is preserved in our model. 
The conditions for the occurrence of these moves are provided by the realization of a graph. 
The local tension tensor is used to determine whether a vertex splits or not. 
We suspend the physical validity of the conditions.

In Section~\ref{section:model}, we introduce two models we call fast and slow deformations. 
Although these models reflect the dependence on the speed of deformation, we consider only the two extreme cases. 
Subsequently, we obtain the stress--strain curve, 
which is merely piecewise continuous. 
We suppose that 
the local moves under deformation finish in finitely many times. 
For example, after a vertex splits, 
the inverse contraction should not occur immediately. 
We show only a sufficient condition to avoid such repetitions.

Section~\ref{section:weight} is devoted to mathematical results. 
In there, 
we consider the extent to which the weight of an edge affects the harmonic realization. 
For the sake of theoretical consideration, 
we allow weights to be non-negative real numbers and continuously deformed. 
When the weight of an edge in a harmonic net becomes large, 
the limit of nets is obtained by the contraction of this edge. 
In Theorem~\ref{thm:loss}, we show a mathematical result 
on the relation between the edge length and the difference of the tension tensor. 
Moreover, in Theorem~\ref{thm:estimate} we obtain a lower bound for the edge length. 

In Section~\ref{section:plastic}, we define the number called the \emph{energy loss ratio} 
to measure the plasticity of nets. 
This provides an estimate of the permanent strain for uniaxial tension. 
We observe simple examples, 
which suggest the followings: 
\begin{enumerate}[(i)]
\item 
a material has lower plasticity if the proportion of loops is large; 
\item 
a material with lower plasticity is obtained by blending two materials. 
\end{enumerate}
Continuous deformation of weights highlights these tendencies.

In Section~\ref{section:ex}, we give examples of deformation. 
We use a periodic graph obtained from the hexagonal lattice. 
The nets obtained by deformation depend on the stretching direction.

\section{Definitions}
\label{section:def}

Based on the formulation in \cite{kotanisunada01,sunada12}, 
we prepare some notions in topological crystallography. 
Let $X = (V,E,w)$ be an (abstract) weighted graph,
which is defined by the vertex set $V$ and the edge set $E$ 
with maps $o, t \colon E \to V$, and $\iota \colon E \to E$ 
such that $\iota^{2} = id$, $\iota (e) \neq e$, and 
$o(\iota (e)) = t(e)$ for any $e \in E$. 
The maps $o$ and $t$ associate the origin and the terminal of an edge, respectively. 
The map $\iota$ reverses the orientation of an edge. 
We allow a loop (edge $e$ such that $o(e) = t(e)$) 
and a multi-edge (edges with common terminal points). 
The weight function $w \colon E \to \mathbb{R}_{\geq 0}$ 
satisfies $w (\iota (e)) = w(e)$ for $e \in E$. 
We regard the weight of an edge as the number of edges. 
Hence, we may replace an edge $e_{0}$ with the union of edges $e_{1}$ and $e_{2}$ 
if $o(e_{0}) = o(e_{1}) = o(e_{2}), t(e_{0}) = t(e_{1}) = t(e_{2})$, and $w(e_{0}) = w(e_{1}) + w(e_{2})$. 
The weight function is often omitted in the notation. 
The degree of a vertex $v \in V$ is defined by $\deg (v) = \sum_{o(e)=v} w(e)$. 
Note that the weight of a loop contributes twice to the degree of its endpoint.

A graph $X=(V,E)$ is a finite graph if $V$ and $E$ are finite sets. 
Otherwise, $X$ is an infinite graph. 
We can naturally identify $X$ with a one-dimensional complex. 
Note that two elements $e$ and $\iota (e)$ in $E$ correspond to 
a single 1-cell in the complex. 
We may reduce the complex by removing the zero-weight edges. 
If this reduced complex is connected, we say that the graph is connected.

We consider an infinite connected graph $X$.
For $N \geq 1$, suppose that $L = \mathbb{Z}^{N}$ acts on $X$ 
as (weight-preserving) automorphisms of the graph, 
the quotient map $\omega \colon X \to X/L = (V/L, E/L)$ is a covering, 
and $X/L$ is a finite graph. 
Then, we say that $X$ is a \emph{periodic graph}, 
and $L$ is a \emph{period lattice} for $X$. 
A map $\Phi \colon V \to \mathbb{R}^{N}$ 
is called a \emph{periodic realization} of $X$ in $\mathbb{R}^{N}$ 
if there exists an injective homomorphism 
$\rho \colon L \hookrightarrow \mathbb{R}^{N}$ as $\mathbb{Z}$-modules 
satisfying that 
\begin{enumerate}[(i)]
	\item $\Phi (\gamma v) = \Phi (v) + \rho (\gamma)$ for any $v \in V$ and $\gamma \in L$, and 
	\item $\rho (L)$ is a lattice subgroup of $\mathbb{R}^{N}$.
\end{enumerate}
Condition (i) means that $\Phi$ is $L$-equivariant.   
We call $\rho$ and $\rho (L)$, respectively, 
the \emph{period homomorphism} and the \emph{period lattice} for $\Phi$.

\begin{dfn}
The pair $(X, \Phi)$ is called as a \emph{net} in $\mathbb{R}^{N}$ 
if $\Phi$ is a periodic realization of a periodic graph $X$ in $\mathbb R^{N}$. 
\end{dfn}

A periodic realization $\Phi$ maps an edge $e \in E$ to a vector 
\[
\bm{v}_{\Phi}(e) = \Phi (t(e)) - \Phi (o(e))
\]
in $\mathbb{R}^{N}$. 
Since $\bm{v}_{\Phi} (\gamma e) = \bm{v}_{\Phi} (e)$ for $\gamma \in L$, 
we obtain $\bm{v}_{\Phi} \colon E/L \to \mathbb{R}^{N}$ 
by $\bm{v}_{\Phi} (\omega (e)) = \bm{v}_{\Phi} (e)$. 
We often write $\Phi$ instead of $\bm{v}_{\Phi}$ by abuse of notation. 
If $e$ is a loop, then $\Phi (e) = 0$.

We define the energy of a net 
and consider energy-minimizing realizations. 
Note that our definition of the energy is slightly different from that in \cite[\S 7.4]{sunada12}, 
where the energy normalized by the volume is defined.

\begin{dfn}
\label{dfn:energy}
The \emph{energy} (per period) of a net $(X, \Phi)$ is defined as follows: 
\[
\mathcal{E}(X, \Phi) = 
\frac{1}{2}\sum_{e \in E/L} w(e) \| \Phi (e) \|^{2}. 
\]
In other words, when we regard edges as springs, the energy is two times the total potential energy of linear springs with rest length equal to zero and elasticity constant given by the edge weight. 
Note that we count the segment between points $P$ and $Q$ twice in the summation,
as edges from $P$ to $Q$ and from $Q$ to $P$. 
\end{dfn}

\begin{dfn}
\label{dfn:harmonic}
A periodic realization $\Phi$ of $X$ is called \emph{harmonic} 
if the energy $\mathcal{E}(X, \Phi)$ is minimal among the periodic realizations of $X$ 
with the fixed period homomorphism $\rho$. 
Then, we call $(X, \Phi)$ a harmonic net. 
\end{dfn}

\begin{dfn}
A periodic realization $\Phi$ of $X$ is called \emph{standard} 
if the energy $\mathcal{E}(X, \Phi)$ is minimal among the periodic realizations of $X$ 
with the fixed covolume $\vol(\mathbb{R}^{N}/\rho (L))$. 
Then, we call $(X, \Phi)$ a standard net. 
\end{dfn}

Remark that when $\vol(\mathbb{R}^{N}/\rho (L))=\vol(\mathbb{R}^{N}/\tilde{\rho} (L))$, 
there exists a volume preserving linear transformation $A\in SL(N,\mathbb{R})$ satisfying $A\circ\rho = \tilde{\rho}$. 
Therefore a periodic realization is standard if its energy is minimal among
its volume-preserving linear transformations.

Clearly, a standard realization is harmonic. 
A harmonic realization is characterized by a local condition. 
We characterize a standard realization in Section~\ref{section:tension}. 

\begin{thm}[\cite{sunada12} Theorem 7.3]
\label{thm:harmonic}
A periodic realization $\Phi$ is harmonic if and only if 
$\sum_{o(e)=v} w(e) \Phi (e) = 0$ for any $v \in V$. 
\end{thm}

From this theorem, it directly follows that 
a linear transformation of a harmonic representation is also harmonic.

\begin{cor}
\label{cor:harmonic}
Suppose that 
$\Phi$ is a periodic realization 
and $A \in GL(N, \mathbb{R})$ is a linear transformation
Then the composition $A \circ \Phi$ is a harmonic realization 
if and only if $\Phi$ is harmonic.
\end{cor}

\begin{rem}
One might think that in Definition~\ref{dfn:energy}, 
the rest lengths of the springs should be positive, not zero. 
However, we have assumed the rest lengths to be zero for two reasons. 
The first reason is due to statistical mechanics. 
A chain in a TPE is not taut like a helical spring, 
but it fills space randomly. 
Therefore, the tension on the chain is proportional to the distance between the endpoints. 
The second reason is a mathematical one. 
We will transform harmonic nets by continuous linear transformations in the following sections. 
However, in order for the result of any linear transformation to be harmonic again, 
the natural length of the spring must be zero 
(see Corollary~\ref{cor:harmonic}). 
\end{rem}

\begin{rem}
In this paper, we use the term `net' as a periodically realized network 
in $\mathbb{R}^N$. 
In the book of Wells~\cite{wells77}, 
who initiated a systematic study of crystal structures as networks, 
a connected simple periodic graph with straight edges in a Euclidean space 
was called a net, and we follow this convention. 
Note that in the terminology of \cite{sunada12}, 
a net is called a topological crystal. 
\end{rem}

\section{Tension tensor}
\label{section:tension}

In our mathematical model, a net represents the structure of TPE chains. 
Consider the tension caused by the structure. 
Indeed, a stretched TPE must have tension in the direction in which the structure is stretched. 
For example, the net in Figure~\ref{fig:HK_symmetricTT.pdf} has a symmetric shape. 
Thus, it has no tension in any direction. 
In contrast, the net in Figure~\ref{fig:HK_stretchedTT.pdf} seems to be stretched 
from the top right to the bottom left. 
However, what can we say about a more complicated net such as that 
in Figure~\ref{fig:HK_complicatedTT.pdf}?

To answer this question, we introduce a matrix named a \emph{tension tensor}. 
Essentially, the tension tensor denotes the energy of a net 
with information of the direction along which the net is stretched, or its `directed energy'. 
We will observe that the tension tensor can be visualized as an ellipse (or ellipsoid), 
such as the ellipses 
in Figures~\ref{fig:HK_symmetricTT.pdf}, \ref{fig:HK_stretchedTT.pdf}, and \ref{fig:HK_complicatedTT.pdf}.

Through computer simulation for the deformation of two-dimensional nets, we make the following observations: 
\begin{enumerate}[(i)]
	\item when we stretch a net and the graph structure does not experience any change, 
	the ellipse of the tension tensor also stretches in the same direction; and 
	\item when the graph structure changes, 
	the ellipse of the tension tensor becomes round. 
	%In this event the energy of the graph decreases. 
\end{enumerate}
The former immediately follows from the definition. 
The latter can be verified by Theorem~\ref{thm:loss}.

\twofig{width=6cm}{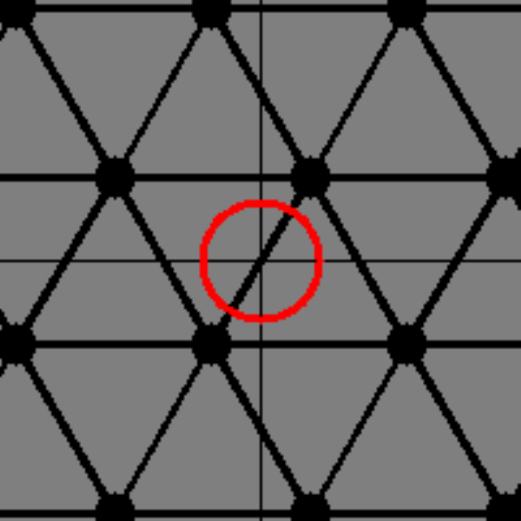}{A symmetric net with tension ellipse}{width=6cm}{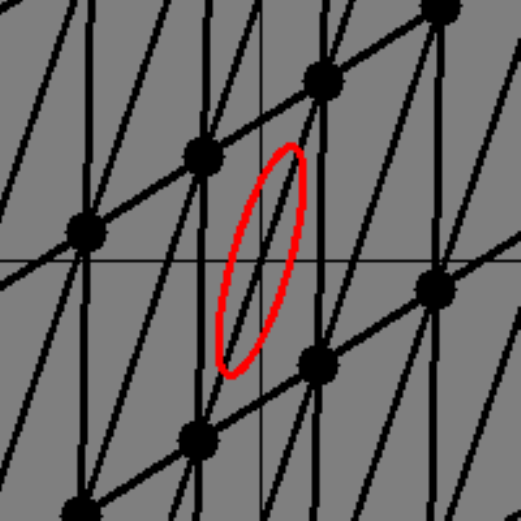}{A stretched net with tension ellipse}
\fig[width=11cm]{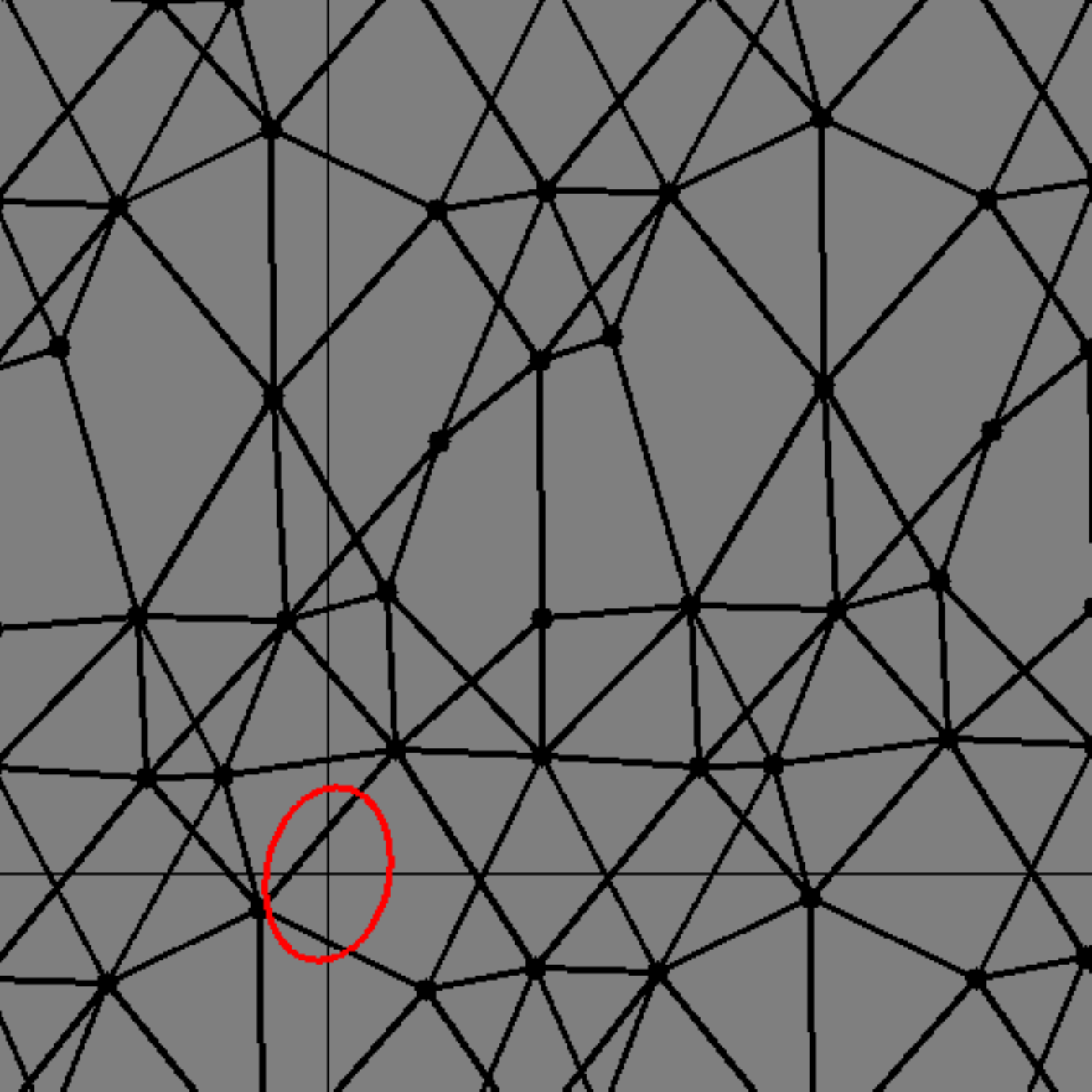}{A complicated net with tension ellipse}

\subsection{Definition of the tension tensor}

\begin{dfn}
For a net $(X,\Phi)$, we define the local and global \emph{tension tensors} as follows: 
For a vertex $v$ of $X$ or $X/L$, 
the local tension tensor is defined by 
\[
\mathcal{T}(v) = \sum_{o(e)=v} w(e) \Phi (e)^{\otimes 2}, 
\]
where 
\[
\begin{pmatrix}
x_{1} \\ \vdots \\ x_{N}
\end{pmatrix}^{\otimes 2}
=
\begin{pmatrix}
x_{1} \\ \vdots \\ x_{N}
\end{pmatrix}
\otimes 
\begin{pmatrix}
x_{1} \\ \vdots \\ x_{N}
\end{pmatrix}
= 
\begin{pmatrix}
x_{1} \\ \vdots \\ x_{N}
\end{pmatrix}
(x_{1}, \dots , x_{N}) = 
\begin{pmatrix}
x_{1}^{2} & \cdots & x_{1}x_{N} \\
\vdots & \ddots & \vdots \\
x_{N}x_{1} & \cdots & x_{N}^{2}
\end{pmatrix}.
\] 
The global tension tensor (per period) is defined by 
\[
\mathcal{T}(X, \Phi) = \frac{1}{2} \sum_{v \in V/L} \mathcal{T}(v). 
\]
\end{dfn}

\begin{prop}
\label{prop:trace}
It holds that 
$\tr(\mathcal{T}(X, \Phi)) = \mathcal{E}(X, \Phi)$.
\end{prop}
\begin{proof}
\begin{align*}
\tr(\mathcal{T}(X, \Phi)) & = 
\frac{1}{2} \sum_{v \in V/L} \sum_{o(e)=v} w(e) \tr(\Phi (e)^{\otimes 2}) \\ 
& = \frac{1}{2} \sum_{e \in E/L} w(e) \| \Phi (e) \|^{2} \\
& = \mathcal{E}(X, \Phi). 
\end{align*}
Note that two elements $e, \iota (e) \in E/L$ are distinguished. 
\end{proof}

The following characterization of a standard realization follows from \cite[Theorem 7.5]{sunada12}. 

\begin{thm}
\label{thm:standard}
A periodic realization $\Phi$ is standard if and only if it is harmonic 
and the global tension tensor $\mathcal{T}(X, \Phi)$ is a constant multiple of the identity matrix. 
\end{thm}

Consequently, a standard realization is unique up to similar transformations. 
The existence and explicit constructions of a standard realization 
were also shown previously \cite{kotanisunada01,sunada12}.

We remark that the global tension tensor $\mathcal{T}(X, \Phi)$ per period 
depends on the choice of period. 
Suppose that $L_{2}$ is a finite index sublattice of $L_{1}=L$, 
and $\mathcal{T}_{i}$ is the tension tensor with respect to the lattice $L_{i}$. 
Then 
\[
\mathcal{T}_{2}(X, \Phi) = [L_{1}:L_{2}] \mathcal{T}_{1}(X, \Phi). 
\]
To avoid this ambiguity, we can define the tension tensor per weight by
\[
\mathcal{T}_{w}(X, \Phi) = \frac{\mathcal{T}(X, \Phi)}{\frac{1}{2} \sum_{e \in E/L} w(e)}. 
\]
However, in most parts of this paper, 
we assume that the covolumes of period lattices are constant, 
and we use the tension tensor per period without the ambiguity.

\subsection{Linear action and visualization}

Let $A \in GL(N,\mathbb{R})$.
The matrix $A$ acts on a net $(X,\Phi)$ by 
\[
A(X,\Phi) = (X, A \circ \Phi). 
\]
Since $\bm{x} \otimes \bm{x} = \bm{x} \, \bm{x}^\mathsf{T}$ for $\bm{x} \in \mathbb{R}^{N}$, 
it follows that $\mathcal{T}(A(X,\Phi)) = A \mathcal{T}(X,\Phi) A^\mathsf{T}$, 
where $\bm{x}^\mathsf{T}$ and $A^\mathsf{T}$ are 
the transposes of $\bm{x}$ and $A$, respectively. 
In particular, if $A$ is a symmetric matrix, 
then $\mathcal{T}(A(X,\Phi)) = A \mathcal{T}(X,\Phi) A$. 

\bigskip

To visualize the tension tensor, we define an ellipsoid by 
\[
\Ell(X,\Phi) = 
\{ \bm{x} \in \mathbb{R}^N \mid 
\bm{x}^\mathsf{T} \mathcal{T}_{w}(X,\Phi)^{-1} \bm{x} = 1 \}. 
\]
We remark that we use the tension tensor per weight here to avoid ambiguity. 
It is easy to check that $\Ell(A(X,\Phi))=A\Ell(X,\Phi)$.

\section{Stress}
\label{section:stress}

In this section, we consider the stress experienced by a net 
by using the tension tensor. 
Fix a periodic graph $X$. 
Let $\Phi$ be a harmonic realization of $X$. 
We regard the energy $\mathcal{E} = \mathcal{E}(X,\Phi)$ as physical energy. 
This is interpreted as the Helmholtz free energy for entropic elasticity. 
With a three-dimensional object in mind, 
we give an obvious generalization to the $N$-dimensional version. 
(See \cite{gurtin81} for the classical theory on continuum mechanics.) 

As a result, the stress satisfies the neo-Hookean model, 
which is the simplest one among the hyperelastic materials. 
This also coincides with the consequence of the classical theory on rubber elasticity by Kuhn 
(see \cite[Ch. 4]{treloar75}). 
Note that his setting is not identical to ours. 
Although polymers are normally distributed in his theory, 
the net in our setting is not isotropic. 
Nonetheless, a standard net has isotropy at the macro-scale.

\subsection{Stress tensor}

By compositing rotational isometry, 
we may assume that 
the tension tensor per period is a diagonal matrix 
$\mathcal{T} = \mathcal{T}(X,\Phi) = \diag(\tau_{1}, \dots , \tau_{N})$. 
Then, the energy per period is $\mathcal{E} = \tr \mathcal{T} = \sum_{i}\tau_{i}$. 
Let $\mathcal{V} = \vol(\mathbb{R}^{N}/\rho(L))$ denote the volume per period.

We apply a physical argument to define stress for nets. 
Let us consider a macro-scale object 
of which the shape is an orthotope ($N$-cuboid) of edge length $L_{i}$ in each $i$-th direction 
for $1 \leq i \leq N$. 
Suppose that this object consists of a net at the micro-scale. 
We apply the affine deformation assumption~\cite[Ch. 4]{treloar75} 
(or the Cauchy-Born rule~\cite{ericksen08cauchy}). 
In other words, 
if the macro-scale object undergoes an affine deformation, 
the net at the micro-scale undergoes the same affine deformation. 
Then, the total energy is equal to 
$(\prod_{i}L_{i} / \mathcal{V}) \mathcal{E} = (\prod_{i}L_{i} / \mathcal{V}) \sum_{i}\tau_{i}$. 
Suppose that external force $F_{i}$ extends outward in each $i$-th direction, 
and the object remains in equilibrium. 
Then, the stress in the $i$-th direction is $\sigma_{i} = L_{i}F_{i} / \prod_{j}L_{j}$. 
Consider infinitesimal deformation of the object. 
For a short while, we allow the volume $\mathcal{V}$ to vary 
but let $\prod_{i}L_{i} / \mathcal{V}$ be constant. 
Let $\Delta L_{i}$ denote the displacement in the $i$-th direction. 
The strain in the $i$-th direction is $\epsilon_{i} = \Delta L_{i} / L_{i}$. 
Then, the work is $\sum_{i} F_{i}\Delta L_{i} = \prod_{j}L_{j} \sum_{i}\sigma_{i}\epsilon_{i}$, 
which is equal to the difference of energies 
$(\prod_{j}L_{j} / \mathcal{V}) \Delta \mathcal{E}$. 
Hence, $\sum_{i}\sigma_{i}\epsilon_{i} = \Delta \mathcal{E} / \mathcal{V}$. 
The difference of the tension tensor is given by 
\[
\Delta \mathcal{T} = 
\diag \left(\tau_{1} (1+\epsilon_{1})^{2}, \dots , \tau_{N} (1+\epsilon_{N})^{2} \right)
- \mathcal{T} = 
\diag(2\tau_{1}\epsilon_{1}, \dots , 2\tau_{N}\epsilon_{N})
\] 
modulo the order more than one. 
Hence, $\Delta \mathcal{E} = \tr (\Delta \mathcal{T}) = \sum_{i} 2\tau_{i}\epsilon_{i}$. 
Therefore, 
\[
\sum_{i} \sigma_{i} \epsilon_{i} = \sum_{i} \frac{2\tau_{i}}{\mathcal{V}} \epsilon_{i}. 
\]

If we can vary $\epsilon_{i}$ freely, 
we obtain $\sigma_{i} = 2\tau_{i} / \mathcal{V}$. 
Thus, we define the \emph{Cauchy stress tensor} for a net as 
$\Sigma = (2 /\mathcal{V}) \mathcal{T}$, 
which is valid in general coordinates.

Furthermore, we suppose that 
the deformation preserves the volume. 
In other words, $\prod_{i}L_{i}$ and $\mathcal{V}$ are constant. 
Since $\prod_{i} (L_{i} + \Delta L_{i}) = \prod_{i}L_{i}$, 
we have $\sum_{i}\epsilon_{i} = 0$. 
If we vary $\epsilon_{i}$ under this condition, 
the equation $\sum_{i}\sigma_{i}\epsilon_{i} = \sum_{i} (2\tau_{i} / \mathcal{V}) \epsilon_{i}$ 
implies that $\sigma_{i} = 2\tau_{i} / \mathcal{V} - c$ for some constant $c$. 
Indeed, uniform pressure does not change the shape under the constraint of volume. 
Thus, the traceless part of the Cauchy stress tensor 
$\Sigma - (\tr(\Sigma) / N) I = 
(2 / \mathcal{V}) \mathcal{T} - (2\mathcal{E} / N\mathcal{V}) I$ 
is regarded as the volume-preserving part. 
This is called the \emph{deviatoric stress tensor}.

If the external forces $F_{i}$ are equal to zero, 
then $2\tau_{i} / \mathcal{V} = c$. 
Hence, the deviatoric stress tensor is zero. 
Since $\mathcal{T} = (c \mathcal{V} / 2) I$, 
Theorem~\ref{thm:standard} implies that $\Phi$ is a standard realization.

A material is hyperelastic (or Green elastic) 
if the stress under deformation is determined by a strain energy density function. 
In our setting, 
consider the affine deformation of a standard net $(X,\Phi)$ by the diagonal matrix 
\[
A = \diag (\lambda_{1}, \dots, \lambda_{N}) \in SL(N, \mathbb{R}). 
\]
The strain energy density function is given by 
\[
\frac{1}{\mathcal{V}} \left( \mathcal{E}(A(X,\Phi)) - \mathcal{E}(X,\Phi) \right) = 
\frac{\mathcal{E}(X,\Phi)}{N\mathcal{V}} \left( \sum_{i}\lambda_{i}^{2} - N \right). 
\]
A material with such a strain energy density function is called incompressible neo-Hookean.

\subsection{Uniaxial extension}

Consider a harmonic net $(X,\Phi)$. 
For the sake of the argument in Section~\ref{section:plastic}, 
first let $\Phi$ be not necessarily standard. 
We write $(\tau_{ij})_{ 1\leq i,j \leq N} = \mathcal{T}(X, \Phi)$, 
which is not necessarily diagonal, in contrast to the previous subsection. 
For $\lambda >0$, the diagonal matrix 
\[
A(\lambda) = \diag \left( \lambda, \lambda^{-1/(N-1)}, \dots, \lambda^{-1/(N-1)} \right) 
\in SL(N, \mathbb{R}) 
\]
induces a uniaxial extension with strain $\epsilon = \lambda -1$. 
The volume $\mathcal{V}$ per period is constant under deformation. 
Consider the tension tensor 
$\mathcal{T}(\lambda) = \mathcal{T}(A(\lambda)(X, \Phi)) 
= A(\lambda) \mathcal{T}(X, \Phi) A(\lambda)$ 
after deformation. 
A stress tensor in the volume-preserving setting is given by 
$(\sigma_{ij})_{ 1\leq i,j \leq N} = (2 / \mathcal{V}) \mathcal{T}(\lambda) - cI$ for some $c$. 
By considering the nature of uniaxial extension, 
we suppose that $\sigma_{22} + \dots + \sigma_{NN} = 0$. 
However, it does not hold that $\sigma_{22} = \dots = \sigma_{NN} = 0$ in general. 
Then, 
\[
c = \frac{2}{(N-1)\mathcal{V}} (\tau_{22} + \dots + \tau_{NN})\lambda^{-2/(N-1)}.
\] 
The \emph{true stress} under this uniaxial extension is defined by 
\[
\sigma_{\mathrm{true}} = \sigma_{11} = \dfrac{2}{\mathcal{V}}\tau_{11}\lambda^{2} - c 
= \dfrac{2}{\mathcal{V}} \left( \tau_{11}\lambda^{2} - \dfrac{1}{N-1} (\tau_{22} + \dots + \tau_{NN})\lambda^{-2/(N-1)} \right). 
\]  
The \emph{engineering stress} (or nominal stress) is measured 
using the cross-sectional area before deformation, 
and it is defined by $\sigma_{\mathrm{eng}} = \sigma_{\mathrm{true}} / \lambda$.

\begin{prop}
\label{prop:stress}
Let $\mathcal{E}(\lambda) = \mathcal{E}(A(\lambda)(X, \Phi))$. 
Then 
\[
\sigma_{\mathrm{true}} = \frac{\lambda}{\mathcal{V}} \frac{d\mathcal{E}(\lambda)}{d\lambda}, \ 
\sigma_{\mathrm{eng}} = \frac{1}{\mathcal{V}} \frac{d\mathcal{E}(\lambda)}{d\lambda}. 
\]
\end{prop}
\begin{proof}
Since $\mathcal{E}(\lambda) = \tr(\mathcal{T}(\lambda)) 
= \tau_{11}\lambda^{2} + (\tau_{22} + \dots + \tau_{NN})\lambda^{-2/(N-1)}$, 
we have 
\[
\frac{d\mathcal{E}(\lambda)}{d\lambda} = 
2 \left( \tau_{11}\lambda - \frac{1}{N-1} (\tau_{22} + \dots + \tau_{NN})\lambda^{-1-2/(N-1)} \right). 
\]
\end{proof}

The \emph{permanent strain} is the number $\epsilon_{0}$ satisfying 
$\sigma_{\mathrm{eng}}(1+\epsilon_{0}) = 0$. 
The following equality is clear from the definition of $\sigma_{\mathrm{eng}}$. 

\begin{prop}
\label{prop:permanent}
It holds that 
\[
\epsilon_{0} = \left( \frac{\tau_{22} + \dots + \tau_{NN}}{(N-1)\tau_{11}} \right)^{(N-1)/2N} -1. 
\]
\end{prop}

Consider the case in which $\Phi$ is standard. 
Then, $\tau_{11}  = \dots = \tau_{NN} = \mathcal{E}(X, \Phi) / N$. 
Moreover, we have $\sigma_{22} = \dots = \sigma_{NN} = 0$, 
which is natural for uniaxial extension. 
The true stress is given by 
\[
\sigma_{\mathrm{true}} = \sigma_{11} 
= \frac{2\mathcal{E}(X, \Phi)}{N\mathcal{V}} \left( \lambda^{2} - \lambda^{-2/(N-1)} \right). 
\]
The engineering stress is given by
\[
\sigma_{\mathrm{eng}} = \frac{\sigma_{\mathrm{true}}}{\lambda} 
= \frac{2\mathcal{E}(X, \Phi)}{N\mathcal{V}} \left( \lambda - \lambda^{-1-2/(N-1)} \right). 
\]
Then, \emph{Young's modulus} for a standard net $(X,\Phi)$ is defined by 
\[
E = \left. \frac{d\sigma_{\mathrm{true}}}{d\lambda} \right|_{\lambda=1} = \frac{4\mathcal{E}(X, \Phi)}{(N-1)\mathcal{V}}. 
\]

\section{Local moves}
\label{section:move}

We introduce three local moves for nets: contraction and splitting. 
A local move for a graph is an operation to obtain a new graph by replacing some vertices and edges. 
When we say that we replace an edge $e$ with $e^{\prime}$, 
we simultaneously replace $\iota (e)$ with $\iota (e^{\prime})$ 
in our notation. 
For a periodic graph, a local move is regarded as an equivariant operation 
preserving the period. 
Even though local moves are defined as operations for abstract graphs, 
the conditions under which they occur are given by realizations of nets.

\subsection{Contraction of two vertices}

Let $X = (V,E)$ be an (abstract) graph, and let $v_{0}, v_{1} \in V$. 
We construct a new graph $X^{\prime} = (V^{\prime}, E^{\prime})$ as follows: 
We define $V^{\prime} = (V \setminus \{v_{0},v_{1}\}) \sqcup v^{\prime}$ and $E^{\prime} = E$. 
Let $\pi \colon V \to V^{\prime}$ denote the projection such that 
$\pi (v_{0}) = \pi (v_{1}) =  v^{\prime}$ 
and the restriction $\pi |_{V \setminus \{v_{0},v_{1}\}}$ is the identity map. 
We define the endpoint maps 
$o^{\prime} = \pi \circ o, t^{\prime} = \pi \circ t \colon E^{\prime} \to V^{\prime}$. 
Suppose that the weight function on $E^{\prime}$ is identical to that on $E$. 
We call this operation the \emph{contraction} of $v_{0}$ and $v_{1}$ to $v^{\prime}$. 
This contraction causes the edges $e_{0}, e_{1}$, and $e_{01}$ to change into
the loop $e^{\prime}$ on $v^{\prime}$, 
where $e_{i}$ is the loop on $v_{i}$, and $e_{01}$ is the edge between $v_{0}$ and $v_{1}$ 
(Figure~\ref{fig:contract.pdf}). 
Then, the sum of weights is preserved.

\fig[width=12cm]{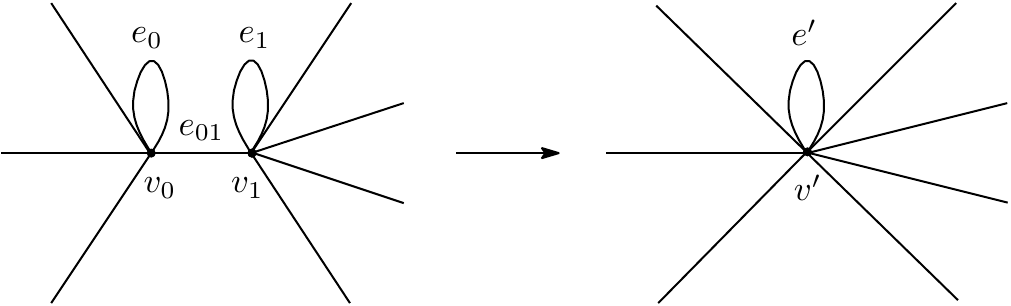}{Contraction of two vertices}

For a periodic graph $X$ with period $L$, 
we define a contraction as an equivariant operation. 
In other words, 
we apply the contraction of $\gamma v_{0}$ and $\gamma v_{1}$ for each $\gamma \in L$. 
In this case, it is necessary that 
$v_{1} \neq \gamma v_{0}$ for any $\gamma \in L$. 
As a result, we obtain a new periodic graph $X^{\prime}$.

To introduce deformation in Section~\ref{section:model}, 
we define a condition for contraction using a realization $\Phi$ of the graph $X$. 
Fix a constant $\delta > 0$. 
If $\| \Phi (v_{1}) - \Phi (v_{0}) \| \leq \delta$, 
then we suppose that the vertices $v_{0}$ and $v_{1}$ contract.

Note that the weight $w(e_{01})$ may be zero. 
Even in this case, 
the vertices $v_{0}$ and $v_{1}$ contract if their distance is sufficiently small.

\subsection{Splitting of a vertex}

A splitting of a vertex is an inverse operation of contraction. 
This is not determined only by the vertex. 
When a vertex $v$ splits, 
the loop on $v$ changes the loops on $v_{0}$ and $v_{1}$ and the edge between them. 
Although the sum of weights is preserved, the choice of the three weights is not unique. 
We also need to assign an endpoint $v_{0}$ or $v_{1}$ 
for each new edge corresponding to an edge originating from $v$. 
Note that there may not necessarily exist loops on the vertex $v$. 
If there exist no loops on $v$, 
there exist no edges between the new vertices $v_{0}$ and $v_{1}$.

For a periodic graph $X$, 
we define a splitting as an equivariant operation 
to obtain a new periodic graph $X^{\prime}$. 
We remark that vertices $v$ and $\gamma v$ for $\gamma \in L$ may be adjacent. 
Even in this case, we can still define a splitting as such. 
By ignoring the period, 
we apply successive splittings of $v$ and $\gamma v$. 
Note that the sequence of splittings in any order yields the same result.

We define a condition for splitting using a realization $\Phi$ of the graph $X$: 
\begin{enumerate}[(i)]
\item when a vertex $v$ splits; 
\item the way in which the edges originating from $v$ are divided into two classes; and 
\item the way in which the weights are assigned. 
\end{enumerate}
Fix constants $K_{d} > 0$ for $d>0$. 
The value $K_{d}$ is regarded as the firmness of a vertex with degree $d$. 
Recall that $\Phi \colon E \to \mathbb{R}^{d}$ is the map induced by $\Phi$, 
and $\mathcal{T}(v) = \sum_{o(e)=v}  w(e) \Phi (e)^{\otimes 2}$ is the local tension tensor around $v$. 
Let $\lambda_{\mathrm{max}}$ denote the maximal eigenvalue of $\mathcal{T}(v)$. 
Take an eigenvector $\bm{u}$ associated with $\lambda_{\mathrm{max}}$. 
We divide the edges originating from $v$ into two classes 
$\{e_{0,j}\}_{1 \leq j \leq m}$ and $\{e_{1,j}\}_{1\leq j \leq n}$ 
so that $\bm{u} \cdot \Phi (e_{0,j}) \leq 0$ and $\bm{u} \cdot \Phi (e_{1,j}) \geq 0$. 
If $\lambda_{\mathrm{max}} \geq K_{\deg (v)}$, then we suppose that the vertex $v$ splits. 
This may be regarded as the maximal principal stress criterion. 
The new edge corresponding to $e_{i,j}$ originates from $v_{i}$. 
Roughly speaking, the splitting occurs in the stretched direction.

To obtain the unique division into two classes $\{e_{0,j}\}$ and $\{e_{1,j}\}$, 
we need the following condition of genericity: 
\begin{enumerate}[(i)]
\item the eigenspace associated with $\lambda_{\mathrm{max}}$ is 
         the one-dimensional space $\Span(\bm{u})$, and 
\item there are no edges $e_{i,j}$ such that $\bm{u} \cdot \Phi (e_{i,j}) = 0$. 
\end{enumerate}

Because it is difficult to decide how the weights are assigned, we use an ad hoc setting: 
Suppose that 
$e_{0}^{\prime}$ and $e_{1}^{\prime}$ are respectively the loops on $v_{0}$ and $v_{1}$. 
Moreover, suppose that $e_{01}^{\prime}$ is the edge between $v_{0}$ and $v_{1}$. 
Fix the probabilities $p_{0}$, $p_{1}$, and $p_{01}$ that 
the loop $e$ on $v$ changes into the new edges 
$e_{0}^{\prime}$, $e_{1}^{\prime}$, and $e_{01}^{\prime}$. 
In other words, 
$w(e_{0}^{\prime}) = p_{0} w(e)$, $w(e_{1}^{\prime}) = p_{1} w(e)$, 
$w(e_{01}^{\prime}) = p_{01} w(e)$, and $p_{0} + p_{1} + p_{01} = 1$. 
Although the choice of $p_{0}$, $p_{1}$, and $p_{01}$ may be arbitrary, 
it is reasonable to set $p_{0} = p_{1} = 1/4$ and $p_{01} = 1/2$. 
The reason is that the above choice holds 
if each endpoint of a new edge is $v_{0}$ with a probability of $1/2$.

For a realization $\Phi$ of a graph $X$, 
suppose that a graph $X^{\prime}$ is obtained by 
splitting a vertex $v$ into $v_{0}$ and $v_{1}$. 
Then, we define the \emph{immediate realization} $\Phi^{(i)}$ of $X^{\prime}$ 
(or $\Phi$ by abuse of notation) 
as follows: 
$\Phi^{(i)} (v_{0}) = \Phi^{(i)} (v_{1}) = \Phi (v)$, 
and $\Phi^{(i)} (u) = \Phi (u)$ for any other vertex $u$. 
If $\Phi$ is a periodic realization, 
then $\Phi^{(i)}$ is equivariantly defined as a periodic realization. 
Using this, we can show that splitting decreases the energy.

\begin{prop}
Suppose that $X^{\prime}$ is a graph obtained from $X$ as a result of splitting 
under the above condition. 
Let $\Phi^{\prime}$ be a harmonic realization of $X^{\prime}$ with the same period as $\Phi$. 
Then, $\mathcal{E}(X^{\prime}, \Phi^{\prime}) < \mathcal{E}(X, \Phi)$. 
\end{prop}
\begin{proof}
Clearly, $\mathcal{E}(X^{\prime}, \Phi^{(i)}) = \mathcal{E}(X, \Phi)$. 
The condition of splitting and Theorem~\ref{thm:harmonic} imply that 
$\Phi^{(i)}$ is not harmonic. 
Hence, $\mathcal{E}(X^{\prime}, \Phi^{\prime}) < \mathcal{E}(X^{\prime}, \Phi^{(i)})$ 
by Definition~\ref{dfn:harmonic}. 
\end{proof}

\begin{rem}
In the above two subsections, 
the graph $X^{\prime}$ obtained by the local move is well-defined. 
However, the realization $\Phi^{\prime}$ of $X^{\prime}$ should be defined using harmonicity, 
so there remains ambiguity of parallel translation. 

In this paper, it does not matter
because we only consider the shape of the realized graph and its energy. 
However, if one wants to discuss such as the displacements of nodes 
before and after a local move, this ambiguity should be removed. 
One idea is to assume that 
the centre of mass of the periodic cell is fixed. 
\end{rem}

\section{Models of deformation}
\label{section:model}

We introduce two models: 
fast and slow deformation. 
The difference of these two models reflects the strain rate sensitivity, 
that is, the dependency of stress on the speed of deformation. 
A harmonic realization of a periodic graph is regarded as an equilibrium state. 
Let $(X_{0}, \Phi_{0})$ be a standard net with period homomorphism $\rho_{0}$ 
as an initial condition. 
This is regarded as a state without external force 
as explained in Section~\ref{section:stress}. 
In this section, we express the period homomorphisms explicitly. 
Suppose that the initial net $(X_{0}, \Phi_{0}, \rho_{0})$ does not satisfy 
any condition of a contraction or splitting. 
Deformation is obtained by linear transformations with constant volume. 
We apply contractions and splittings that satisfy the conditions in Section~\ref{section:move} 
for harmonic nets in deformation. 
Because the process is not deterministic in general, 
it is necessary to choose one that satisfies the conditions. 
We leave the stochastic formulation for future work.

\subsection{Fast deformation} 
Let $A \in SL(N, \mathbb{R})$. 
Fix the period homomorphism $A \circ \rho_{0}$. 
We apply local moves for the harmonic net 
$A(X_{0}, \Phi_{0}, \rho_{0}) = (X_{0}, A \circ \Phi_{0}, A \circ \rho_{0})$. 
First, we apply the splittings. 
Then, we obtain a new harmonic net. 
If the conditions of other local moves hold, 
we continue to apply the splittings. 
Second, we apply the contractions. 
However, more than two vertices may contract to a point. 
In general, we need to choose the contracting vertices 
so that the contraction does not violate the period.

We continue this procedure by 
supposing that these procedures finish with finitely many local moves. 
In the end, 
we obtain a harmonic net $(X_{1}, \Phi_{1}, A \circ \rho_{0})$. 
We call this process \emph{fast deformation}.

%If the conditions of contractions and splittings hold simultaneously, we suppose that the splittings occur earlier. 

\subsection{Slow deformation}

Slow deformation is a limit of sequences of fast deformation. 
For a continuous family of linear transformations, 
take approximations by discrete families of small ones. 
They induce sequences of fast deformation. 
We obtain slow deformation by the limit as the approximations get arbitrarily fine.

Equivalently and more precisely, slow deformation is defined as follows. 
Suppose that 
$A_{t} \in SL(N, \mathbb{R})$ for $0 \leq t \leq 1$ is 
a continuous family of linear transformations such that $A_{0} = \mathrm{id}$. 
Let $\rho_{t} = A_{t} \circ \rho_{0}$. 
We apply local moves while increasing $t$ from zero to one. 
Let $t_{1}$ be the minimal $t$ such that 
the condition of a local move holds for the harmonic net $(X_{0}, A_{t} \circ \Phi_{0}, \rho_{t})$. 
We obtain a graph $X_{0}^{\prime}$ by the local move. 
Consider a harmonic net $(X_{0}^{\prime}, \Phi_{0}^{\prime}, \rho_{t_{1}})$. 
Note that another local move may occur for $(X_{0}^{\prime}, \Phi_{0}^{\prime}, \rho_{t_{1}})$. 
Then, we continue to apply local moves. 
Subsequently, we obtain a harmonic net $(X_{t_{1}}, \Phi_{t_{1}}, \rho_{t_{1}})$. 

After the exhaustion of local moves for $t_{1}$, 
we increase $t$. 
Let $t_{2}$ be the minimal $t$ more than $t_{1}$ such that 
the condition of a local move holds for the harmonic net 
$(X_{t_{1}}, A_{t}A_{t_{1}}^{-1} \circ \Phi_{t_{1}}, \rho_{t_{1}})$. 
Using the same argument as above, 
we obtain a harmonic net $(X_{t_{2}}, \Phi_{t_{2}}, \rho_{t_{2}})$. 

We continue this procedure by 
supposing that these procedures finish with finitely many local moves. 
Ultimately, 
we obtain a harmonic net $(X_{1}, \Phi_{1}, \rho_{1})$. 
We call this process \emph{slow deformation}.

A condition of genericity is given as follows: 
\begin{enumerate}[(i)]
\item two local moves do not occur simultaneously, and 
\item two vertices equivalent by the period do not contract. 
\end{enumerate}
If the net is highly symmetric, 
the genericity is difficult to hold. 
For genericity, we arbitrarily choose a single local move at a time, 
and we ignore any contraction that violates the period.

%(perturb the harmonic realization or the weights of edges) 
%(sequential destruction vs sporadic destruction)

\bigskip

We define the \emph{stress--strain curve} for a uniaxial extension. 
Let 
\[
A(\lambda) = \diag \left( \lambda, \lambda^{-1/(N-1)}, \dots , \lambda^{-1/(N-1)} \right) 
\in SL(N, \mathbb{R}). 
\]
The slow deformation by $A_{t} = A(\lambda^{t})$ for $0 \leq t \leq 1$ 
induces a net $(X_{1}, \Phi_{1})$. 
The energy $\mathcal{E}(\lambda) = \mathcal{E}(X_{1}, \Phi_{1})$ 
is a right-continuous function of $\lambda$. 
We can plot the stress--strain curve 
as a graph of the engineering stress $\sigma_{\mathrm{eng}}$ 
as a function of the strain $\epsilon = \lambda-1$, 
where $\sigma_{\mathrm{eng}} = \dfrac{1}{\mathcal{V}} \dfrac{d\mathcal{E}(\lambda)}{d\lambda}$ 
by Proposition~\ref{prop:stress}.

\subsection{Compatibility of splitting and contraction}

We suppose that 
the above procedures in deformation finish with finitely many local moves. 
In general, splittings and contractions may cause an infinite sequence of local moves. 
We give only a partial result for this problem. 

%We want a sufficient condition that local moves finish with finitely many times. 

Let $(X, \Phi)$ be a harmonic net. 
Let $X^{\prime}$ be a periodic graph obtained from $X$ 
by splitting a vertex $v$ into $v_{0}$ and $v_{1}$ 
with respect to the condition given in Section~\ref{section:move}. 
The maximal eigenvalue of the local tension tensor $\mathcal{T}(v)$ is equal to $K_{\deg (v)}$. 
Let $e^{\prime}$ denote the new non-loop edge in $X^{\prime}$. 
Let $w^{\prime}$ be the weight of $e^{\prime}$, 
which does not exceed the weight of the loop on $v$. 
We consider a harmonic realization $\Phi^{\prime}$ 
with the same period lattice as $\Phi$. 
If $\|\Phi^{\prime}(e^{\prime})\| = \|\Phi^{\prime}(v_{1}) - \Phi^{\prime}(v_{0})\| \leq \delta$, 
then splittings and contractions continue alternately. 
We show that such repetition does not occur 
if $\delta$ is sufficiently small.

\begin{thm}
\label{thm:compatibility}
Suppose that the weights are non-negative integers. 
Then, 
\[
\|\Phi^{\prime}(e^{\prime})\|
\geq \frac{\sqrt{2 K_{\deg (v)}}}{\deg (v)}. 
\]
\end{thm}
\begin{proof}
Let $\Phi^{(a)}$ be an auxiliary periodic realization of $X^{\prime}$ 
such that $\Phi^{(a)}(v_{0}) = \Phi (v)$, 
$\Phi^{(a)}(u) = \Phi (u)$ for any vertex $u \neq \gamma v_{1}$ ($\gamma \in L$), 
and $\Phi^{(a)}$ is harmonic around $v_{1}$. 
Then, $\|\Phi^{\prime}(e^{\prime})\| \geq \|\Phi^{(a)}(e^{\prime})\|$, 
which we show in Theorem~\ref{thm:estimate}. 
Hence, it is sufficient to show that 
\[
\|\Phi^{(a)}(e^{\prime})\|
\geq \frac{1}{\deg (v_{1})} \sqrt{\frac{K_{\deg (v)}}{2}}. 
\]
Indeed, 
since $\deg (v) = \deg (v_{0}) + \deg (v_{1})$, 
we may assume that $\deg (v) \geq 2\deg (v_{1})$ 
by interchanging $v_{0}$ and $v_{1}$ if necessary.

Let $e_{i1}, \dots , e_{in_{i}}$ for $i=0,1$ denote the non-loop edges of $X^{\prime}$ 
originating from $v_{i}$ other than $e^{\prime}$, 
and let $v_{i1}, \dots, v_{in_{i}}$ denote their terminals. 
Note that we may ignore the edges between $v_{i}$ and $\gamma v_{i}$ for $\gamma \in L$. 
Let $w_{ij}$ be the weight of $e_{ij}$. 
The same symbol is used for the corresponding edges and vertices of $X$. 
We write $\Phi (v_{ij}) = (x^{1}_{ij}, \dots , x^{N}_{ij})$. 
We may assume that 
$\Phi (v) = \Phi^{(a)} (v_{0}) = 0$ 
and the splitting occurs in the direction of the first coordinate; 
that is, the vector $(1,0, \dots, 0)$ is an eigenvector associated with the maximal eigenvalue 
of the local tension tensor around $v$. 
Then, $x^{1}_{0j} \leq 0$, $x^{1}_{1j} \geq 0$, 
$\sum_{i,j} w_{ij} x^{1}_{ij} = 0$, 
and $\sum_{i,j} w_{ij}(x^{1}_{ij})^{2} = K_{\deg (v)}$. 
Since $\Phi^{(a)}$ is harmonic around $v_{1}$, 
we have 
\[
-w^{\prime} \Phi^{(a)}(v_{1}) + \sum_{j=1}^{n_{1}} w_{1j} (\Phi (v_{1j}) - \Phi^{(a)}(v_{1})) = 0. 
\]
Hence, 
\[
\|\Phi^{(a)} (e^{\prime})\| 
= \|\Phi^{(a)} (v_{1})\| 
= \left\| \frac{\sum_{j=1}^{n_{1}} w_{1j} \Phi (v_{1j})}{w^{\prime} + \sum_{j=1}^{n_{1}} w_{1j}} \right\|
\geq \frac{\sum_{j=1}^{n_{1}} w_{1j} x^{1}_{1j}}{\deg (v_{1})}. 
\]
Moreover, 
we obtain $\sum_{j=1}^{n_{1}} w_{1j} x^{1}_{1j} \geq \sqrt{K_{\deg (v)} / 2}$ 
by Lemma~\ref{lem:min} and the assumption that $w_{ij} \in \mathbb{Z}_{\geq 0}$. 
\end{proof}

\begin{lem}
\label{lem:min}
Let $x_{01}, \dots , x_{0n_{0}} \leq 0$ and 
$x_{11}, \dots , x_{1n_{1}} \geq 0$. 
Suppose that $z = -\sum_{j=1}^{n_{0}} x_{0j} = \sum_{j=1}^{n_{1}} x_{1j}$ 
and $K = \sum_{i,j} (x_{ij})^{2}$. 
Then, $z \geq \sqrt{K / 2}$. 
\end{lem}
\begin{proof}
We find the maximum $K$ for a fixed $z>0$. 
If $x + y = a$ is fixed for $x,y \geq 0$, then 
the maximum $a^{2}$ of $x^{2}+y^{2}$ is attained when $x=0$ or $y=0$. 
Hence, the maximum of $K$ is attained 
when $x_{0j_{0}} = -z$ and $x_{1j_{1}} = z$ for some $j_{0}$ and $j_{1}$, 
and $x_{ij} = 0$ for the other $j$. 
Therefore, $K \leq 2z^{2}$. 
\end{proof}

\section{Variation of weights}
\label{section:weight}

In this section, we consider the extent to which the harmonic realizations and their energies depend on the weights, 
which varies in non-negative real numbers. 
We describe a contraction as the limit by increasing the weight of an edge. 
Note that we do not suppose that the sum of weights is preserved, which differs from the assumption in Section~\ref{section:move}. 

Let $X=(V,E)$ be a periodic graph with period $L$. 
We take representatives $v_{0}, v_{1}, \dots, v_{n} \in V$ of the set $V/L$. 
Let $e_{ij\gamma}$ denote the edge from $v_{i}$ to $\gamma v_{j}$ for $\gamma \in L$. 
Then, $\{e_{ij\gamma}\}$ are representatives of the set $E/L$. 
Let $w_{ij\gamma}$ denote the weight of $e_{ij\gamma}$. 
Then, $w_{ij,-\gamma} = w_{ji\gamma}$. 

Fix a period homomorphism $\rho \colon L \to \mathbb{R}^{N}$. 
Suppose that $n \geq 1$. 
Let $\hat{X}$ be a periodic graph obtained from $X$ 
by the contraction of $v_{0}$ and $v_{1}$ to a vertex $\hat{v}$. 
Suppose that $\Phi^{(h)}$ and $\hat{\Phi}^{(h)}$ are 
harmonic realizations of $X$ and $\hat{X}$, respectively. 
We may assume that $\Phi^{(h)}(v_{0}) = \hat{\Phi}^{(h)}(\hat{v}) = 0$. 
We change the harmonic realizations $\Phi^{(h)}$ 
by varying $w_{010}$ while fixing the other weights on $X$. 
Here, we regard $w_{ij\gamma} = w_{ji,-\gamma}$ as a single variable. 
Suppose that $X$ is connected, 
that is, the union of its edges with positive weights is connected. 
Since the harmonic realization $\Phi^{(h)}$ is given by 
the unique solution of a system of linear equations, 
it depends continuously on $w_{010}$. 
After we show that $\hat{\Phi}^{(h)}$ can be regarded as 
$\lim\limits_{w_{010} \to \infty} \Phi^{(h)}$, 
we give explicit presentations.

\begin{lem}
\label{lem:limit}
The realizations $\Phi^{(h)}$ converge to $\hat{\Phi}^{(h)}$ as $w_{010} \to \infty$. 
In other words, $\Phi^{(h)}(v_{i})$ converge to $\hat{\Phi}^{(h)}(v_{i})$ 
for each $2 \leq i \leq n$, 
and $\Phi^{(h)}(v_{1})$ converge to $\hat{\Phi}^{(h)}(\hat{v}) = 0$. 
In particular, $\lim\limits_{w_{010} \to \infty} \Phi^{(h)}(e_{010}) = 0$. 
Consequently, 
\[
\lim\limits_{w_{010} \to \infty} \mathcal{T}(X,\Phi^{(h)}) 
= \mathcal{T}(\hat{X},\hat{\Phi}^{(h)}) 
\ \text{and} \
\lim\limits_{w_{010} \to \infty} \mathcal{E}(X,\Phi^{(h)}) 
= \mathcal{E}(\hat{X},\hat{\Phi}^{(h)}). 
\]
\end{lem}
\begin{proof}
Consider a (not necessarily harmonic) periodic realization $\Phi$ of $X$. 
We write $\Phi (v_{i}) = (x^{1}_{i}, \dots , x^{N}_{i}) \in \mathbb{R}^{N}$ 
and $\rho (\gamma) = (\rho^{1}_{\gamma}, \dots , \rho^{N}_{\gamma}) \in \mathbb{R}^{N}$ 
for $\gamma \in L$. 
Then, $\rho^{k}_{-\gamma} = -\rho^{k}_{\gamma}$. 
By Theorem~\ref{thm:harmonic}, 
the realization $\Phi$ is harmonic if and only if 
\begin{equation}
\label{eq:harmonic}
\sum_{j=0}^{n} \sum_{\gamma \in L} w_{ij\gamma} (-x^{k}_{i}+x^{k}_{j}+\rho^{k}_{\gamma}) = 0
\end{equation} 
for any $0 \leq i \leq n$ and $1 \leq k \leq N$. 
Let $x^{k}_{i} = \xi^{k}_{i}$ be a solution of this system of equations, 
which is unique up to translations. 
We may assume that $\xi^{k}_{0} = 0$. 
Then, $\Phi^{(h)}(v_{i}) = (\xi^{1}_{i}, \dots , \xi^{N}_{i})$. 
We write $w_{ij} = \sum_{\gamma} w_{ij\gamma}$, 
$b_{ij} = -w_{ij}$ for $i \neq j$, 
$b_{ii} = \sum_{j \neq i} w_{ij}$, and 
$c^{k}_{i} = \sum_{j,\gamma} w_{ij\gamma}\rho^{k}_{\gamma}$. 
Using the matrix $B = (b_{ij})$, 
the system of linear equations is written as 
$B (x^{k}_{0}, \dots, x^{k}_{n})^{\mathsf{T}}
= (c^{k}_{0}, \dots, c^{k}_{n})^{\mathsf{T}}$. 
Since $\xi^{k}_{0} = 0$ and $(\xi^{k}_{i})_{1 \leq i \leq n}$ are unique, 
we have 
$B_{00} (\xi^{k}_{1}, \dots, \xi^{k}_{n})^{\mathsf{T}}
= (c^{k}_{1}, \dots, c^{k}_{n})^{\mathsf{T}}$,
where $B_{00} = (b_{ij})_{1 \leq i,j \leq n}$ is a minor of $B$. 
Then, $B_{00}$ is invertible. 
Cramer's rule implies that 
$\xi^{k}_{i} = \det C^{k}_{i} / \det B_{00}$, 
where $C^{k}_{i}$ is the matrix obtained by replacing the $i$-th column of $B_{00}$ 
with $(c^{k}_{1}, \dots, c^{k}_{n})^{\mathsf{T}}$. 
In this presentation of $\det C^{k}_{i} / \det B_{00}$, 
only $b_{11} = \sum_{j \neq 1,\gamma} w_{1j\gamma}$ 
contains $w_{010}=w_{100}$. (Note that $\rho^{k}_{0} = 0$.) 
Hence, $\det B_{00}$ and $\det C^{k}_{i}$ are linear functions of $w_{010}$. 
Moreover, $\det C^{k}_{1}$ is constant for $w_{010}$.

Since $\Phi^{(h)}$ is harmonic, 
we obtain $\mathcal{E}(X,\Phi^{(h)}) \leq \mathcal{E}(\hat{X},\hat{\Phi}^{(h)})$ 
by regarding $\hat{\Phi}^{(h)}$ as a realization of $X$. 
Moreover, we have $w_{010} \|\Phi^{(h)}(e_{010})\|^{2} \leq \mathcal{E}(X,\Phi^{(h)})$. 
Hence, $\lim\limits_{w_{010} \to \infty} \Phi^{(h)}(e_{010}) = 0$. 
In other words, $\lim\limits_{w_{010} \to \infty} \xi^{k}_{1} = 0$. 
If $\det B_{00}$ is constant for $w_{010}$, 
then $\xi^{k}_{1} = \det C^{k}_{1} / \det B_{00}$ is also constant. 
Hence, $\xi^{k}_{1} = 0$. 
Then, $\Phi^{(h)} = \hat{\Phi}^{(h)}$, and the assertion holds trivially. 

Suppose that $\det B_{00}$ is not constant for $w_{010}$. 
Then, $\xi^{k}_{i} = \det C^{k}_{i} / \det B_{00}$ converges as $w_{010} \to \infty$. 
We define the realization $\hat{\Phi}$ of $\hat{X}$ such that 
$\hat{\Phi} (\hat{v}) = 0$ and 
$\hat{\Phi} (v_{i}) = \lim\limits_{w_{010} \to \infty} \Phi^{(h)} (v_{i})$ for $2 \leq i \leq n$. 
Since $\mathcal{E}(\hat{X},\hat{\Phi}) = \lim\limits_{w_{010} \to \infty} \mathcal{E}(X,\Phi^{(h)}) 
\leq \mathcal{E}(\hat{X},\hat{\Phi}^{(h)})$ 
and $\hat{\Phi}^{(h)}$ is harmonic, 
the realization $\hat{\Phi}$ is also harmonic. 
The uniqueness of a harmonic realization implies that $\hat{\Phi} = \hat{\Phi}^{(h)}$. 
Therefore, $\lim\limits_{w_{010} \to \infty} \Phi^{(h)} = \hat{\Phi}^{(h)}$. 
\end{proof}

\begin{thm}
\label{thm:loss}
There are $\bm{z} \in \mathbb{R}^{N}$ and $W \in \mathbb{R}$ such that 
\[
\Phi^{(h)}(e_{010}) = \dfrac{\bm{z}}{w_{010}+W}
\] 
and 
\[
\mathcal{T}(\hat{X},\hat{\Phi}^{(h)}) - \mathcal{T}(X,\Phi^{(h)}) 
= \dfrac{\bm{z}^{\otimes 2}}{w_{010}+W}, 
\]
where $\bm{z}$ and $W$ do not depend on $w_{010}$, 
and $W$ does not depend on $\rho$. 
Consequently, 
\[
\mathcal{E}(\hat{X},\hat{\Phi}^{(h)}) - \mathcal{E}(X,\Phi^{(h)}) 
= \dfrac{\|\bm{z}\|^{2}}{w_{010}+W}
= (w_{010}+W) \| \Phi^{(h)}(e_{010}) \|^{2}. 
\]
\end{thm}

\begin{proof}
We showed that $\xi^{k}_{1} = \det C^{k}_{1} / \det B_{00}$ 
in the proof of Lemma~\ref{lem:limit}. 
The determinants $\det B_{00}$ and $\det C^{k}_{1}$ are, respectively, linear and constant 
as functions of $w_{010} = w_{100}$. 
If $\det B_{00}$ is constant for $w_{010}$, 
then $\xi^{k}_{1} = 0$, and we obtain $\bm{z} = 0$. 
We can take $W$ arbitrarily.

Suppose that $\det B_{00}$ is not constant for $w_{010}$. 
Then, we can write $\det B_{00} = (w_{010} + W) \det B_{00,11}$, 
where $B_{00,11} = (b_{ij})_{2 \leq i,j \leq n}$, 
and $W$ does not depend on $w_{010}$ or $\rho$. 
Let 
$z^{k} = \det C^{k}_{1} / \det B_{00,11}$ 
and $\bm{z} = (z^{1}, \dots , z^{N})$. 
Then, 
\[
\Phi^{(h)}(e_{010}) = (\xi^{1}_{1}, \dots , \xi^{N}_{1}) = \dfrac{\bm{z}}{w_{010} + W}. 
\]
Moreover, $\bm{z}$ does not depend on $w_{010}$.

We regard the tension tensor $\mathcal{T}(X,\Phi)$ 
as a function of $w = (w_{ij\gamma})$ and $x = (x_{i}^{k})$. 
Its $(k,l)$-entry is given by 
\[
\mathcal{T}^{kl}(w,x) = \frac{1}{2} \sum_{i,j=0}^{n} \sum_{\gamma \in L}
w_{ij\gamma} (-x^{k}_{i} + x^{k}_{j} + \rho^{k}_{\gamma}) (-x^{l}_{i} + x^{l}_{j} + \rho^{l}_{\gamma}). 
\]
Then, 
\[
\dfrac{\partial \mathcal{T}^{kl}}{\partial x^{\beta}_{\alpha}} (w,x) = 
-\delta_{\beta k} \sum_{j,\gamma} w_{\alpha j\gamma} 
(-x^{l}_{\alpha}+x^{l}_{j}+\rho^{l}_{\gamma}) 
-\delta_{\beta l} \sum_{j,\gamma} w_{\alpha j\gamma} 
(-x^{k}_{\alpha}+x^{k}_{j}+\rho^{k}_{\gamma}). 
\]
Consider $\mathcal{T}^{kl}(w) = \mathcal{T}^{kl}(w, \xi(w))$ as a function of $w$, 
where $\xi(w) = (\xi^{k}_{i}(w))$. 
By the equality (\ref{eq:harmonic}), we have 
$\dfrac{\partial \mathcal{T}^{kl}}{\partial x^{\beta}_{\alpha}}(w, \xi(w)) = 0$. 
Hence, 
\begin{align*}
\frac{\partial \mathcal{T}^{kl}}{\partial w_{010}} (w)
& = \frac{\partial \mathcal{T}^{kl}}{\partial w_{010}}(w, \xi(w)) + 
\sum_{\alpha, \beta} \frac{\partial \mathcal{T}^{kl}}{\partial x^{\beta}_{\alpha}}(w, \xi(w)) 
\frac{\partial \xi^{\beta}_{\alpha}}{\partial w_{010}}(w) \\ 
& = \frac{\partial \mathcal{T}^{kl}}{\partial w_{010}}(w, \xi(w)) \\ 
& = \xi^{k}_{1} \xi^{l}_{1} \\
& = \dfrac{z^{k} z^{l}}{(w_{010} + W)^{2}}, 
\end{align*}
where we regard $w_{010} = w_{100}$ as a single variable. 
Hence, 
$\mathcal{T}^{kl} = C - z^{k} z^{l} / (w_{010} + W)$ 
for some $C$ independent of $w_{010}$. 
Since 
$\lim\limits_{w_{010} \to \infty} \mathcal{T}(X,\Phi^{(h)}) = \mathcal{T}(\hat{X},\hat{\Phi}^{(h)})$ 
by Lemma~\ref{lem:limit}, 
we have 
\[
\mathcal{T}(\hat{X},\hat{\Phi}^{(h)}) - \mathcal{T}(X,\Phi^{(h)}) 
= \left( \dfrac{z^{k} z^{l}}{w_{010} + W} \right)_{1 \leq k,l \leq N} 
= \dfrac{\bm{z}^{\otimes 2}}{w_{010}+W}. 
\]
\end{proof}

\begin{lem}
\label{lem:weight}
Let the vector $\bm{z}$ and the number $W$ be as in Theorem~\ref{thm:loss}. 
Suppose that $\bm{z} \neq 0$. 
Then, 
\[
0 < W \leq \left( \sum_{j \neq 1}\sum_{\gamma \in L} w_{1j\gamma} \right) - w_{100}.  
\] 
\end{lem}
\begin{proof}
Since 
\[
\dfrac{\|\bm{z}\|^{2}}{w_{010}+W} 
= \mathcal{E}(\hat{X},\hat{\Phi}^{(h)}) - \mathcal{E}(X,\Phi^{(h)}) 
< \mathcal{E}(\hat{X},\hat{\Phi}^{(h)}) 
\]
for any $w_{010} \geq 0$, 
we have $W > 0$. 

As in the proof of Theorem~\ref{thm:loss}, 
we have $w_{010} + W = \det B_{00} / \det B_{00,11}$, where 
$w_{ij} = \sum_{\gamma} w_{ij\gamma}$, 
$b_{ij} = -w_{ij}$ for $i \neq j$, 
$b_{ii} = \sum_{j \neq i} w_{ij}$,
$B_{00} = (b_{ij})_{1 \leq i,j \leq n}$, 
and $B_{00,11} = (b_{ij})_{2 \leq i,j \leq n}$. 
Then, 
\[
\dfrac{\det B_{00}}{\det B_{00,11}} = b_{11} - 
(b_{12}, \dots , b_{1n})
B_{00,11}^{-1}
(b_{12}, \dots , b_{1n})^{\mathsf{T}}
\]
by Lemma~\ref{lem:expansion}. 
Since $B_{00,11}$ is positive definite 
by Lemma~\ref{lem:posdef} and $\det B_{00,11} \neq 0$, 
so is $B_{00,11}^{-1}$. 
Therefore, $w_{010} + W \leq b_{11} = \sum_{j \neq 1}\sum_{\gamma \in L} w_{1j\gamma}$. 
Note that if $n=1$, we conventionally set $\det B_{00,11} =1$ and $w_{010} + W = b_{11}$. 
\end{proof}

\begin{lem}
\label{lem:expansion}
Let $A = (a_{ij})_{1 \leq i,j \leq n}$ be a symmetric matrix. 
Suppose that $A_{11} = (a_{ij})_{2 \leq i,j \leq n}$ is invertible. 
Then 
\[
\frac{\det A}{\det A_{11}} = 
a_{11} - 
(a_{12}, \dots , a_{1n})
A_{11}^{-1}
(a_{12}, \dots , a_{1n})^{\mathsf{T}}.
\]
\end{lem}
\begin{proof}
Using the adjugate matrix, we have 
\[
A_{11}^{-1} = 
(\det A_{11})^{-1} \left( (-1)^{i+j}\det A_{11,ji} \right)_{2 \leq i,j \leq n}, 
\]
where $A_{11,ij} = (a_{kl})_{k\neq 1,i, l \neq 1,j}$. 
The cofactor expansion implies that 
\begin{align*}
\det A 
& = a_{11} \det A_{11} 
+ \sum_{2 \leq i,j \leq n} (-1)^{i+j+1} a_{i1}a_{1j} \det A_{11,ij} \\
& = a_{11} \det A_{11} - 
(a_{12}, \dots , a_{1n})
(\det A_{11}) A_{11}^{-1}
(a_{12}, \dots , a_{1n})^{\mathsf{T}}.
\end{align*}
\end{proof}

\begin{lem}
\label{lem:posdef}
Let $A = (a_{ij})_{1 \leq i,j \leq n}$ be a symmetric matrix. 
Suppose that $a_{ij} \leq 0$ for any $i \neq j$ and 
$\sum_{j=1}^{n} a_{ij} \geq 0$ for any $i$. 
Then, $A$ is positive semi-definite. 
\end{lem}
\begin{proof}
The proof is by induction on $n$. 
The assertion is trivial for the case $n=1$. 
We have 
\[
A = 
\diag \left( \sum_{j=1}^{n} a_{1j}, \dots , \sum_{j=1}^{n} a_{nj} \right)
+ P^{\mathsf{T}} A^{\prime} P, 
\]
where 
\[
A^{\prime} = 
\begin{pmatrix}
a_{11} & \cdots & a_{1,n-1} & 0 \\
\vdots & \ddots & \vdots & \vdots \\
a_{n-1,1} & \cdots & a_{n-1,n-1} & 0 \\
0 & \cdots & 0 & 0 
\end{pmatrix}
- \diag \left( \sum_{j=1}^{n} a_{1j}, \dots, \sum_{j=1}^{n} a_{n-1,j}, 0 \right)
\]
and 
\[
P = 
\begin{pmatrix}
1      & \cdots & 0      & -1 \\
\vdots & \ddots & \vdots & \vdots \\
0      & \cdots & 1      & -1 \\
0      & \cdots & 0      & 1 
\end{pmatrix}.
\]
The matrix $A^{\prime}$ is positive semi-definite by the assumption of induction for $n-1$. 
Therefore, $A$ is also positive semi-definite. 
\end{proof}

\begin{thm}
\label{thm:estimate}
Let $\Phi^{(a)}$ be a realization of $X$ such that 
$\Phi^{(a)}(v_{i}) = \hat{\Phi}^{(h)}(v_{i})$ for any $i \neq 1$ 
and $\Phi^{(a)}$ is harmonic around $v_{1}$. 
Then, $\|\Phi^{(h)}(e_{010})\| \geq \|\Phi^{(a)}(e_{010})\|$. 
\end{thm}
\begin{proof}
Since $\Phi^{(a)}$ is harmonic around $v_{1}$ and $\sum_{\gamma \in L} w_{11\gamma} \rho (\gamma) = 0$, 
we have 
\[
\sum_{j \neq 1}\sum_{\gamma \in L} w_{1j\gamma}(-\Phi^{(a)}(v_{1}) + \hat{\Phi}^{(h)}(v_{j}) + \rho (\gamma)) = 0. 
\]
Hence, 
\[
\Phi^{(a)}(v_{1}) = 
\frac{\sum_{j \neq 1}\sum_{\gamma \in L} w_{1j\gamma}(\hat{\Phi}^{(h)}(v_{j}) + \rho (\gamma))}{\sum_{j \neq 1}\sum_{\gamma \in L} w_{1j\gamma}}. 
\]
Let 
\begin{align*}
\bm{z}^{(a)} 
&= \sum_{j \neq 1}\sum_{\gamma \in L} w_{1j\gamma}(\hat{\Phi}^{(h)}(v_{j}) + \rho (\gamma)), \\
W^{(a)} 
&= \left( \sum_{j \neq 1}\sum_{\gamma \in L} w_{1j\gamma} \right) - w_{100}. 
\end{align*}
Then, $\Phi^{(a)}(e_{010}) = \Phi^{(a)}(v_{1}) = \bm{z}^{(a)} / (w_{010}+W^{(a)})$. 
The difference of energies is given by 
\begin{align*}
\mathcal{E}(\hat{X}, \hat{\Phi}^{(h)}) - \mathcal{E}(X, \Phi^{(a)}) 
& = \sum_{j \neq 1}\sum_{\gamma \in L} w_{1j\gamma}\|\hat{\Phi}^{(h)}(v_{j}) + \rho (\gamma)\|^{2} \\
& \quad - \sum_{j \neq 1}\sum_{\gamma \in L} w_{1j\gamma}\|-\Phi^{(a)}(v_{1}) + \hat{\Phi}^{(h)}(v_{j}) + \rho (\gamma)\|^{2} \\
& = -\sum_{j \neq 1}\sum_{\gamma \in L} w_{1j\gamma}\|\Phi^{(a)}(v_{1})\|^{2} \\
& \quad + 2 \sum_{j \neq 1}\sum_{\gamma \in L} w_{1j\gamma} 
(\hat{\Phi}^{(h)}(v_{j}) + \rho (\gamma)) \cdot \Phi^{(a)}(v_{1}) \\
& = -(w_{010} + W^{(a)}) \left\| \frac{\bm{z}^{(a)}}{w_{010}+W^{(a)}} \right\|^{2} 
+ 2\bm{z}^{(a)} \cdot \frac{\bm{z}^{(a)}}{w_{010}+W^{(a)}} \\
& = \frac{\|\bm{z}^{(a)}\|^{2}}{w_{010}+W^{(a)}}. 
\end{align*}

Since $\Phi^{(h)}$ is harmonic, 
we have $\mathcal{E}(X, \Phi^{(h)}) \leq \mathcal{E}(X, \Phi^{(a)})$. 
Lemma~\ref{lem:weight} implies that $W \leq W^{(a)}$. 
Therefore, 
\begin{align*}
\|\Phi^{(h)}(e_{010})\|^{2} 
& = 
\frac{\mathcal{E}(\hat{X}, \hat{\Phi}^{(h)}) - \mathcal{E}(X, \Phi^{(h)})}{w_{010}+W} \\
& \geq 
\frac{\mathcal{E}(\hat{X}, \hat{\Phi}^{(h)}) - \mathcal{E}(X, \Phi^{(a)})}{w_{010}+W^{(a)}} \\
& = \|\Phi^{(a)}(e_{010})\|^{2} 
\end{align*}
by Theorem~\ref{thm:loss}. 
\end{proof}

\section{Plasticity}
\label{section:plastic}

The plasticity of a material is its ability to undergo permanent deformation. 
For a fixed periodic graph, no external force is applied to a net 
if and only if its realization is standard, 
as shown in Section~\ref{section:stress}. 
Hence, if no local moves occur in the deformation of a net, 
then it returns to its initial state by unloading, and it is perfectly elastic. 
In general, however, local moves cause plasticity. 
To measure plasticity, 
we introduce the energy loss ratio of a net under deformation. 
This is defined by comparing the energy with that of a net 
in which local moves do not occur.

\begin{dfn}
Let $(X_{0}, \Phi_{0})$ be a harmonic net in $\mathbb{R}^{N}$. 
Let us consider fast deformation by $A=A_{1} \in SL(N, \mathbb{R})$ 
or slow deformation by $A_{t} \in SL(N, \mathbb{R})$ for $0 \leq t \leq 1$. 
Suppose that we obtain a harmonic net $(X_{1}, \Phi_{1})$ 
as in Section~\ref{section:model}. 
We define the \emph{energy loss ratio} of $X_{0}$ with respect to $A$ or $A_{t}$ as 
\[
R(X_{0}, \Phi_{0}, A_{(t)}) = 
\frac{\mathcal{E}(X_{0},\Phi_{0}) - \mathcal{E}(A_{1}^{-1}(X_{1},\Phi_{1}))}{\mathcal{E}(X_{0},\Phi_{0})}. 
\]
\end{dfn}

If no local moves occur, then $R = 0$. 
We may regard the ratio $R$ as a degree of destruction. 
Note that $R$ may be negative by some occurrence of contractions.

\bigskip

Consider the uniaxial extension by 
\[
A(\lambda) = \diag \left( \lambda, \lambda^{-1/(N-1)}, \dots , \lambda^{-1/(N-1)} \right) 
\in SL(N, \mathbb{R}) 
\]
and $A_{t} = A(\lambda^{t})$. 
The permanent strain $\epsilon_{0}$ was defined in Section~\ref{section:stress}, 
by setting $A_{1}^{-1}(X_{1},\Phi_{1})$ as the reference position.

\begin{prop}
\label{prop:e0vsR}
Suppose that $\Phi_{0}$ is standard, no contractions occur in the deformation, 
and $R = R(X_{0}, \Phi_{0}, A_{t}) < 1/N$. 
Then, the permanent strain $\epsilon_{0}$ satisfies that 
\[
\left( 1 - \frac{N}{N-1}R \right)^{(N-1)/2N} -1 \leq 
\epsilon_{0} \leq (1-NR)^{-(N-1)/2N} -1. 
\]
\end{prop}
Approximations 
\[
\left( 1 - \frac{N}{N-1}R \right)^{(N-1)/2N} -1 \sim -\frac{1}{2}R, \
(1-NR)^{-(N-1)/2N} -1 \sim \frac{N-1}{2}R 
\]
hold when $R \sim 0$. 
One might expect that $\epsilon_{0} \geq 0$ for $\lambda >1$, 
but it does not generally hold, 
because the directions of deformation and splittings 
do not necessarily coincide. 

\begin{proof}
Write $\mathcal{E}_{0} = \mathcal{E}(X_{0},\Phi_{0})$, 
$\mathcal{E}_{1} = \mathcal{E}(A_{1}^{-1}(X_{1},\Phi_{1}))$, and 
$(\tau_{ij}) = \mathcal{T}(A_{1}^{-1}(X_{1},\Phi_{1}))$. 
Then, $\mathcal{T}(X_{0},\Phi_{0}) = (\mathcal{E}_{0} / N) I$, 
$\mathcal{E}_{1} = \tau_{11} + \dots + \tau_{NN}$, and 
$R = 1 - \mathcal{E}_{1} / \mathcal{E}_{0}$. 
Since only splittings occur, we have 
$\tau_{ii} \leq \mathcal{E}_{0} / N$ for each $1 \leq i \leq N$ 
by Theorem~\ref{thm:loss}. 
Hence, 
\[
\tau_{11} = \mathcal{E}_{1} - (\tau_{22} + \dots + \tau_{NN}) 
\geq \mathcal{E}_{1} - \dfrac{N-1}{N}\mathcal{E}_{0}. 
\]

Proposition~\ref{prop:permanent} implies that 
\[
(1 + \epsilon_{0})^{2N/(N-1)} = \frac{\tau_{22} + \dots + \tau_{NN}}{(N-1)\tau_{11}} 
= \frac{\mathcal{E}_{1} - \tau_{11}}{(N-1)\tau_{11}} 
= \frac{1}{N-1} \left( \frac{\mathcal{E}_{1}}{\tau_{11}} - 1 \right). 
\]
Since $\tau_{11} \leq \mathcal{E}_{0} / N$, 
we have 
$\mathcal{E}_{1} / \tau_{11} \geq N\mathcal{E}_{1} / \mathcal{E}_{0} = N(1-R)$. 
Since $R < 1/N$ and $\tau_{11} \geq \mathcal{E}_{1} - ((N-1)/N)\mathcal{E}_{0}$, 
we have 
\[
\dfrac{\mathcal{E}_{1}}{\tau_{11}} \leq \left( 1 -\dfrac{N-1}{N(1-R)} \right)^{-1} 
= \dfrac{N(1-R)}{1-NR}. 
\] 
Therefore, 
\[
1 - \frac{N}{N-1}R \leq (1 + \epsilon_{0})^{2N/(N-1)} \leq \frac{1}{1-NR}. 
\]
\end{proof}

In the remainder of this section, 
we consider the simplest case of splitting; 
let $X_{0}$ be a periodic graph such that only a single vertex exists in each period. 
We identify each $\bm{i} = (i_{1}, \dots , i_{N}) \in \mathbb{Z}^{N}$ with a vertex of $X_{0}$. 
Let $e_{\bm{i}}$ denote the edge of $X_{0}$ joining $0$ and $\bm{i} \in \mathbb{Z}^{N}$. 
We write $w_{\bm{i}}$ for the weight of $e_{\bm{i}}$. 
It is necessary that $w_{\bm{i}} = w_{-\bm{i}}$. 
Let $\rho \colon \mathbb{Z}^{N} \to \mathbb{R}^{N}$ be a period homomorphism, 
and let $(\bm{u}_{1}, \dots , \bm{u}_{N})$ be a basis of $\rho (\mathbb{Z}^{N})$. 
The period homomorphism $\rho$ induces a periodic realization $\Phi_{0}$ of $X_{0}$ 
such that the image of vertices is $\rho (\mathbb{Z}^{N})$. 
For $\bm{i} \in \mathbb{Z}^{N}$, 
the edge $e_{\bm{i}}$ corresponds to 
$\bm{v}_{\bm{i}} = i_{1}\bm{u}_{1} + \dots + i_{N}\bm{u}_{N} \in \mathbb{R}^{N}$. 
Then, $\Phi_{0}$ is a harmonic realization. 
Let $I \subset \mathbb{Z}^{N}$ such that $\mathbb{Z}^{N} = I \sqcup -I \sqcup \{ 0 \}$. 
Let $X_{1}$ be a periodic graph obtained from $X_{0}$ 
by splitting the vertex $0$ into $v_{0}$ and $v_{1}$ 
so that $v_{0}$ and $v_{1}$ are endpoints of $e^{\prime}_{\bm{i}}$ 
for $\bm{i} \in -I$ and $\bm{i} \in I$, respectively,
where $e^{\prime}_{\bm{i}}$ is an edge of $X_{1}$ obtained from $e_{\bm{i}}$. 
We write $w_{0}^{\prime}$ for the weight of the new non-loop edge. 
Let $\Phi_{1}$ be a harmonic realization of $X_{1}$ with the period homomorphism $\rho$. 
We may assume that $\Phi_{1}(v_{0}) = 0$. Let $\bm{x} = \Phi_{1}(v_{1})$.

For $\bm{u} \neq 0 \in \mathbb{R}^{N}$, 
let $\mathbb{R}^{N}_{\bm{u}} = \{ \bm{x} \in \mathbb{R}^{N} \mid \bm{u} \cdot \bm{x} > 0 \}$ 
and $I_{\bm{u}} = \mathbb{R}^{N}_{\bm{u}} \cap \mathbb{Z}^{N}$. 
If $I_{\bm{u}} \subset I$, 
the vertices split in the direction of $\bm{u}$. 
Subsequently, it is possible to regard the net $(X_{1},\Phi_{1})$ as a result of slow deformation. 
However, we do not require this assumption unless otherwise stated.

After we show general behaviour, 
we give some examples of the energy loss ratio 
$R = (\mathcal{E}(X_{0}, \Phi_{0}) - \mathcal{E}(X_{1}, \Phi_{1})) / \mathcal{E}(X_{0}, \Phi_{0})$ 
depending on the weights $\{w_{\bm{i}}\}_{\bm{i} \in \mathbb{Z}^{N}}$.

\begin{prop}
\label{prop:loss1v}
It holds that 
\begin{align*}
\bm{x} 
& = \frac{\sum_{\bm{i} \in I}w_{\bm{i}}\bm{v}_{\bm{i}}}{w_{0}^{\prime} + \sum_{\bm{i} \in I}w_{\bm{i}}}, \\
\mathcal{E}(X_{0},\Phi_{0}) - \mathcal{E}(X_{1},\Phi_{1}) 
& = \frac{\left\|\sum_{\bm{i} \in I}w_{\bm{i}}\bm{v}_{\bm{i}}\right\|^{2}}{w_{0}^{\prime} + \sum_{\bm{i} \in I}w_{\bm{i}}}
= \left( w_{0}^{\prime} + \sum_{\bm{i} \in I}w_{\bm{i}} \right) \| \bm{x} \|^{2}. 
\end{align*}
\end{prop}
\begin{proof}
The condition of harmonic realization 
$-w_{0}^{\prime}\bm{x} + \sum_{\bm{i} \in I} w_{\bm{i}}(\bm{v}_{\bm{i}} - \bm{x}) = 0$ 
implies the presentation of $\bm{x}$. 
We have $w_{010} = w_{0}^{\prime}$, $\bm{z} = \sum_{\bm{i} \in I}w_{\bm{i}}\bm{v}_{\bm{i}}$, 
and $W = \sum_{\bm{i} \in I}w_{\bm{i}}$ in the notation of Theorem~\ref{thm:loss}. 
Thus, we obtain the presentation of $\mathcal{E}(X_{0},\Phi_{0}) - \mathcal{E}(X_{1},\Phi_{1})$. 

We remark that $\Phi_{1}$ coincides with $\Phi^{(a)}$ in Theorem~\ref{thm:estimate}. 
Moreover, $W$ attains the maximum in Lemma~\ref{lem:weight}. 
\end{proof}

By way of example, 
we suppose that the weights are given by a function of the lengths of edges.

\begin{thm}
\label{thm:limratio}
Let $F(\bm{x})$ be a non-negative function on $\mathbb{R}^{N}$ 
such that $F(-\bm{x}) = F(\bm{x})$. 
Put $w_{\bm{i}} = F(s\bm{v}_{\bm{i}})$ for $s>0$. 
Suppose that the sum 
$\sum_{\bm{i} \in I}F(s\bm{v}_{\bm{i}})\|\bm{v}_{\bm{i}}\|^{2}$ is convergent. 
Let $p = w_{0}^{\prime} / w_{0}$, 
which is regarded as the probability that a loop changes into a non-loop edge. 
Suppose that there is $\bm{u} \neq 0 \in \mathbb{R}^{N}$ such that $I_{\bm{u}} \subset I$. 
Then, the energy loss ratio $R(s,p)$ satisfies 
\[
\lim_{s \to 0} R(s,p) = 
\frac{\left\| \int_{\mathbb{R}^{N}_{\bm{u}}} F(\bm{x})\bm{x}d\bm{x} \right\|^{2}}{\int_{\mathbb{R}^{N}_{\bm{u}}} F(\bm{x})d\bm{x} \int_{\mathbb{R}^{N}_{\bm{u}}} F(\bm{x})\|\bm{x}\|^{2}d\bm{x}}, 
\]
where we suppose that the three integrals are finite. 
Moreover, if $F(0)>0$ and 
$\lim\limits_{s \to \infty} \sum_{\bm{i} \in I}F(s\bm{v}_{\bm{i}}) = 0$, 
then $\lim\limits_{s \to \infty} R(s,p) = 0$ for any fixed $p>0$. 
\end{thm}

By the second assertion, 
we may understand that the material has lower plasticity 
if the proportion of loops is large. 

\begin{proof}
Since $\|\bm{v}_{\bm{i}}\| \geq 1$ for all but finitely many $\bm{i}$, 
the sums $\sum_{\bm{i} \in I}F(s\bm{v}_{\bm{i}})$ and 
$\left\|\sum_{\bm{i} \in I}F(s\bm{v}_{\bm{i}})\bm{v}_{\bm{i}}\right\|$ are absolutely convergent. 
The volume per period is given by 
$\mathcal{V} = \det 
\begin{pmatrix}
\bm{u}_{1} & \cdots & \bm{u}_{N} 
\end{pmatrix}$. 
By Proposition~\ref{prop:loss1v}, 
\[
R(s,p) 
= \frac{\left\|\sum_{\bm{i} \in I}F(s\bm{v}_{\bm{i}})\bm{v}_{\bm{i}}\right\|^{2}}{\left( pF(0)+\sum_{\bm{i} \in I}F(s\bm{v}_{\bm{i}}) \right) \sum_{\bm{i} \in I}F(s\bm{v}_{\bm{i}})\|\bm{v}_{\bm{i}}\|^{2}}. 
\]
Hence,
\begin{align*}
\lim_{s \to 0} R(s,p) 
& = \lim_{s \to 0}  \frac{\left\| s^{N}\mathcal{V} \sum_{\bm{i} \in I}F(s\bm{v}_{\bm{i}})s\bm{v}_{\bm{i}}\right\|^{2}}{\left( s^{N}\mathcal{V} pF(0) + s^{N}\mathcal{V} \sum_{\bm{i} \in I}F(s\bm{v}_{\bm{i}}) \right) s^{N}\mathcal{V} \sum_{\bm{i} \in I}F(s\bm{v}_{\bm{i}})\|s\bm{v}_{\bm{i}}\|^{2}} \\
& = \lim_{s \to 0} \frac{ \left\| s^{N}\mathcal{V} \sum_{\bm{i} \in I}F(s\bm{v}_{\bm{i}})s\bm{v}_{\bm{i}}\right\|^{2}}{\left( s^{N}\mathcal{V} \sum_{\bm{i} \in I}F(s\bm{v}_{\bm{i}}) \right) s^{N}\mathcal{V} \sum_{\bm{i} \in I}F(s\bm{v}_{\bm{i}})\|s\bm{v}_{\bm{i}}\|^{2}} \\
& = \frac{\left\| \int_{\mathbb{R}^{N}_{\bm{u}}} F(\bm{x})\bm{x}d\bm{x} \right\|^{2}}{\int_{\mathbb{R}^{N}_{\bm{u}}} F(\bm{x})d\bm{x} \int_{\mathbb{R}^{N}_{\bm{u}}} F(\bm{x})\|\bm{x}\|^{2}d\bm{x}}, 
\end{align*}
where the Riemann sums converge to the Riemann integrals. 

Suppose that $p>0$, $F(0)>0$, and 
$\lim\limits_{s \to \infty} \sum_{\bm{i} \in I}F(s\bm{v}_{\bm{i}}) = 0$. 
Since $R(s,0) \leq 1$, we have 
\begin{align*}
\lim_{s \to \infty} R(s,p) 
& = \lim_{s \to \infty} \frac{\sum_{\bm{i} \in I}F(s\bm{v}_{\bm{i}})}{pF(0) + \sum_{\bm{i} \in I}F(s\bm{v}_{\bm{i}})} R(s,0) \\
& = 0.  
\end{align*}
\end{proof}

%$\sigma$ is its standard deviation. 

For example, we use the normal distribution, 
given by the function 
$F(\bm{x}) = (2\pi \sigma^{2})^{-N/2} \exp \left( -\|\bm{x}\|^{2} / 2\sigma^{2} \right)$ 
for $\bm{x} \in \mathbb{R}^{N}$ and $\sigma > 0$. 
Then, 
\begin{align*}
\lim_{\sigma \to \infty} R(\sigma,p) 
& = \frac{\left\| \int_{\mathbb{R}^{N}_{\bm{u}_{1}}} e^{-\|\bm{x}\|^{2}}\bm{x}d\bm{x} \right\|^{2}}{\int_{\mathbb{R}^{N}_{\bm{u}_{1}}} e^{-\|\bm{x}\|^{2}}d\bm{x}  \int_{\mathbb{R}^{N}_{\bm{u}_{1}}} e^{-\|\bm{x}\|^{2}}\|\bm{x}\|^{2}d\bm{x}} \\
& = \frac{\left( \int_{0}^{\infty} e^{-x^{2}}xdx \left( \int_{\mathbb{R}} e^{-x^{2}}dx \right)^{N-1} \right) ^{2}}{\frac{1}{2} \left( \int_{\mathbb{R}} e^{-x^{2}}dx \right)^{N} 
\frac{N}{2} \int_{\mathbb{R}} e^{-x^{2}}x^{2}dx \left( \int_{\mathbb{R}} e^{-x^{2}}dx \right)^{N-1}} \\
& = \frac{4 \left( \int_{0}^{\infty} e^{-x^{2}}xdx \right) ^{2}}{N \int_{\mathbb{R}} e^{-x^{2}}dx 
\int_{\mathbb{R}} e^{-x^{2}}x^{2}dx} \\
& = \frac{2}{N\pi}, 
\end{align*}
and $\lim_{\sigma \to 0} R(\sigma,p) = 0$ for $p>0$. 
Note that constant multiples of the weights do not change the energy loss ratio $R$. 
Generally, 
if the function $F(s\bm{x})$ for each fixed $\bm{x} \neq 0 \in \mathbb{R}^{N}$ 
monotonically decreases for $s \geq 0$ and $\sum_{\bm{i} \in I} F(s\bm{v}_{\bm{i}}) < \infty$, 
then $\lim\limits_{s \to \infty} \sum_{\bm{i} \in I}F(s\bm{v}_{\bm{i}}) = 0$.

Next, we consider the weight functions given by the linear sums 
$w_{s,\bm{i}} = (1-s)w_{0,\bm{i}} + sw_{1,\bm{i}}$ for $0 \leq s \leq 1$. 
The weight function $w_{s,\bm{i}}$ for each $s$ induces the energy 
\[
\mathcal{E}_{s} = \mathcal{E}((X_{0},w_{s,\bm{i}}),\Phi_{0}) 
= \sum_{\bm{i} \in I} w_{s,\bm{i}} \|\bm{v}_{\bm{i}}\|^{2} 
= (1-s)\mathcal{E}_{0} + s\mathcal{E}_{1}
\] 
and the energy loss ratio 
$R_{s} = \|\bm{z}_{s}\|^{2} / \overline{W}_{s}\mathcal{E}_{s}$ by Proposition~\ref{prop:loss1v}, 
where 
\begin{align*}
\bm{z}_{s} 
& = \sum_{\bm{i} \in I}w_{s,\bm{i}}\bm{v}_{\bm{i}} 
= (1-s)\bm{z}_{0} + s\bm{z}_{1}, \\
\overline{W}_{s} 
& = pw_{s,0} + \sum_{\bm{i} \in I} w_{s,\bm{i}} 
= (1-s)\overline{W}_{0} + s\overline{W}_{1}. 
\end{align*}
We consider the way in which $R_{s}$ depends on the compounding ratio $s$. 
Since $\|\bm{z}_{0}\|^{2} = \overline{W}_{0}\mathcal{E}_{0}R_{0}$ 
and $\|\bm{z}_{1}\|^{2} = \overline{W}_{1}\mathcal{E}_{1}R_{1}$, 
we have 
\[
R_{s} = 
\frac{\|(1-s)\bm{z}_{0} + s\bm{z}_{1}\|^{2}}{\overline{W}_{s}\mathcal{E}_{s}}
= 
\frac{(1-s)^{2}\overline{W}_{0}\mathcal{E}_{0}R_{0} + s^{2}\overline{W}_{1}\mathcal{E}_{1}R_{1} + 2(1-s)s\bm{z}_{0} \cdot \bm{z}_{1}}{\overline{W}_{s}\mathcal{E}_{s}}. 
\]

\begin{prop}
\label{prop:blend}
Suppose that $R_{0} = R_{1}$. 
Then 
\[
R_{s} = R_{0} - \frac{(1-s)s}{\overline{W}_{s}\mathcal{E}_{s}}
\left( (\overline{W}_{0}\mathcal{E}_{1}+\overline{W}_{1}\mathcal{E}_{0})R_{0} - 2\bm{z}_{0} \cdot \bm{z}_{1} \right) \leq R_{0}. 
\]
The equality holds if and only if 
$\mathcal{E}_{0} / \overline{W}_{0} = \mathcal{E}_{1} / \overline{W}_{1}$ 
and the vectors $\bm{z}_{0}$ and $\bm{z}_{1}$ are parallel. 
For fixed $w_{0,\bm{i}}$ and $w_{1,\bm{i}}$, 
the minimum of $R_{s}$ is attained 
when 
\[
s = \hat{s} = \dfrac{\sqrt{\overline{W}_{0}\mathcal{E}_{0}}}{\sqrt{\overline{W}_{0}\mathcal{E}_{0}} + \sqrt{\overline{W}_{1}\mathcal{E}_{1}}}. 
\]
\end{prop}

We may understand that 
a material with lower plasticity is obtained 
by blending two materials. 
Note that 
if $I = I_{\bm{u}}$ for $\bm{u} \neq 0 \in \mathbb{R}^{N}$, 
the vectors $z_{0}$ and $z_{1}$ are likely to be nearly parallel to $\bm{u}$. 

\begin{proof}
It is easy to check the presentation of $R_{s}$. 
The inequality $R_{s} \leq R_{1}$ follows from 
\[
2\bm{z}_{0} \cdot \bm{z}_{1} \leq 2 \|\bm{z}_{0}\| \|\bm{z}_{1}\| 
= 2\sqrt{\overline{W}_{0}\mathcal{E}_{0}\overline{W}_{1}\mathcal{E}_{1}}R_{0} 
\leq (\overline{W}_{0}\mathcal{E}_{1}+\overline{W}_{1}\mathcal{E}_{0})R_{0}, 
\]
where 
the first inequality is the Cauchy--Schwarz inequality, 
and the second follows from 
$(\sqrt{\overline{W}_{0}\mathcal{E}_{1}} - \sqrt{\overline{W}_{1}\mathcal{E}_{0}})^{2} \geq 0$. 
An easy calculation shows that 
\begin{align*}
\frac{(1-s)s}{\overline{W}_{s}\mathcal{E}_{s}} 
& = \frac{(1-s)s}{((1-s)\overline{W}_{0} + s\overline{W}_{1})((1-s)\mathcal{E}_{0} + s\mathcal{E}_{1})} \\
& = \frac{s(1-s)^{-1}}{(\overline{W}_{0} + s(1-s)^{-1}\overline{W}_{1})(\mathcal{E}_{0} + s(1-s)^{-1}\mathcal{E}_{1})}
\end{align*}
increases for $0 < s  < \hat{s}$ and decreases for $\hat{s} < s <1$. 
\end{proof}

We remark that 
the linear sums of weights with different energy loss ratios 
do not yield smaller ratios in general. 
For instance, the linear sums of $w_{0,\bm{i}}$ and $w_{\hat{s},\bm{i}}$ 
as in Proposition~\ref{prop:blend} 
give the energy loss ratios between $R_{0}$ and $R_{\hat{s}}$.

Under the assumption that $R_{0} = R_{1}$, we have 
\[
R_{\hat{s}} = 
\frac{R_{0} + \bm{z}_{0} \cdot \bm{z}_{1} / \sqrt{\overline{W}_{0}\mathcal{E}_{0}\overline{W}_{1}\mathcal{E}_{1}}}{1 + (\overline{W}_{0}\mathcal{E}_{1}+\overline{W}_{1}\mathcal{E}_{0}) / 2\sqrt{\overline{W}_{0}\mathcal{E}_{0}\overline{W}_{1}\mathcal{E}_{1}}} 
\leq 
\frac{2R_{0}}{1 + (\overline{W}_{0}\mathcal{E}_{1}+\overline{W}_{1}\mathcal{E}_{0}) / 2\sqrt{\overline{W}_{0}\mathcal{E}_{0}\overline{W}_{1}\mathcal{E}_{1}}} 
\]
and 
\[
\frac{\overline{W}_{0}\mathcal{E}_{1}+\overline{W}_{1}\mathcal{E}_{0}}{2\sqrt{\overline{W}_{0}\mathcal{E}_{0}\overline{W}_{1}\mathcal{E}_{1}}} 
= \dfrac{x^{2}+1}{2x}, 
\]
where $x = \sqrt{\overline{W}_{0}\mathcal{E}_{1} / \overline{W}_{1}\mathcal{E}_{0}}$. 
If $x$ or $x^{-1}$ is large, then $R_{\hat{s}}$ is nearly equal to zero.

For example, we consider the cube lattice. 
Suppose that $(\bm{u}_{1}, \dots , \bm{u}_{N})$ is the standard basis of $\mathbb{R}^{N}$. 
For $m \in \mathbb{Z}_{>0}$, 
let $w_{0,0} = w_{1,0} = a$, $w_{0,\pm \bm{u}_{k}} = w_{1,\pm m\bm{u}_{k}}=1$ for any $k=1, \dots , N$, 
and $w_{0,\bm{i}} = w_{1,\bm{i}} =0$ for the other $\bm{i}$. 
Let $\bm{u} = (1, \dots , 1)$ and $I = I_{\bm{u}}$. 
Then, 
$\overline{W}_{0} = \overline{W}_{1} = pa+N$, 
$\mathcal{E}_{0} = N$, $\mathcal{E}_{1} = m^{2}N$, 
$\bm{z}_{0} = (1, \dots , 1)$, $\bm{z}_{1} = (m, \dots , m)$, 
and $R_{0} = R_{1} = 1/(pa+N)$. 
Since $\hat{s} = 1/(1+m)$, 
we have $R_{\hat{s}} = 4m(1+m)^{-2}R_{0}$. 
Consequently, $\mathcal{E}_{\hat{s}} = mN$ and $\lim\limits_{m \to \infty} R_{\hat{s}} = 0$.

We give another example 
by using the linear sums of normal distributions. 
Suppose that $w_{k,\bm{i}} = 
(2\pi \sigma_{k}^{2})^{-N/2} 
\exp \left( -\|\bm{v}_{\bm{i}}\|^{2} / 2\sigma_{k}^{2} \right)$ 
for $k=0,1$ and $\sigma_{k} > 0$. 
Fix $\mu = \sigma_{1} / \sigma_{0}$. 
Theorem~\ref{thm:limratio} implies that 
\[
\tilde{R}_{s} = \lim\limits_{\sigma_{0} \to \infty} R_{s} = 
\frac{2(1-s+s\mu)^{2}}{N\pi (1-s+s\mu^{2})}. 
\]
It attains the minimum 
when $s = \hat{s} = 1/(1+\mu)$. 
Hence, $\tilde{R}_{0} = \tilde{R}_{1} = 2/N\pi$, 
$\tilde{R}_{\hat{s}} = 4\mu(1+\mu)^{-2}\tilde{R}_{0}$, 
and $\lim\limits_{\mu \to \infty} \tilde{R}_{\hat{s}} = 0$. 
We remark that 
$\lim\limits_{\sigma_{0} \to \infty} \overline{W}_{1} / \overline{W}_{0} = 1$ and 
$\lim\limits_{\sigma_{0} \to \infty} \mathcal{E}_{1} / \mathcal{E}_{0} = \mu^{2}$.

\section{Examples of deformation}
\label{section:ex}

In this section, we give examples of deformation for two-dimensional nets. 
Let 
\[
A(\lambda) = 
\begin{pmatrix}
\cos \theta & -\sin \theta \\
\sin \theta & \cos \theta
\end{pmatrix}
\begin{pmatrix}
\lambda & 0 \\
0 & \lambda^{-1}
\end{pmatrix}
\begin{pmatrix}
\cos \theta & \sin \theta \\
-\sin \theta & \cos \theta
\end{pmatrix}. 
\] 
We consider a uniaxial extension with strain $\epsilon = \lambda -1$ 
in the direction of an angle $\theta$ from the horizontal axis, 
namely, the slow deformation by $A_{t} = A(\lambda^{t})$.

For the splitting of a vertex $v$ into $v_{0}$ and $v_{1}$, 
let $e, e^{\prime}_{0}, e^{\prime}_{1}$, and $e^{\prime}_{01}$ respectively denote 
the loops on $v,v_{0},v_{1}$ and the non-loop edge between $v_{0}$ and $v_{1}$. 
Suppose that the weight function $w$ satisfies 
$w(e^{\prime}_{0}) = w(e^{\prime}_{1}) = w(e)/4$ 
and $w(e^{\prime}_{01}) = w(e)/2$.

Let $X^{(m)}$ denote the periodic graph after the $m$-th local move, 
and let $\Phi^{(m)}$ be a harmonic realization of $X^{(m)}$. 
Suppose that $\Phi^{(0)}$ is standard, 
and all the period homomorphisms for $\Phi^{(m)}$ are common. 
We write $\mathcal{E}^{(m)}(\lambda) = \mathcal{E}(A(\lambda)(X^{(m)},\Phi^{(m)}))$. 
The engineering stress is 
$\sigma_{\mathrm{eng}}^{(m)}(\lambda) 
= \dfrac{1}{\mathcal{V}} \dfrac{d\mathcal{E}^{(m)}}{d\lambda}$. 
The permanent strain $\epsilon_{0}^{(m)}$ is the number satisfying 
$\sigma_{\mathrm{eng}}^{(m)}(1+\epsilon_{0}^{(m)}) = 0$. 
The energy loss ratio is 
$R^{(m)} = 1 - \mathcal{E}^{(m)}(1) / \mathcal{E}^{(0)}(1)$.

\subsection{Hexagonal lattice}

Let $X^{(0)}$ be a periodic graph of which the edges consist of 
loops and the 1-skeleton of the hexagonal tiling. 
We assume that the period is minimal. 
Let $w_{0}$ and $w_{1}$ be the weights of each loop and non-loop edge, respectively. 
A standard realization $\Phi^{(0)}$ of $X^{(0)}$ is given as follows 
(Figure~\ref{fig:hex0.pdf}): 
\begin{enumerate}[(i)]
\item for $l>0$, the vectors 
	$\bm{u}_{1} = (\sqrt{3}l,0)$ and $\bm{u}_{2} = ((\sqrt{3}/2) l, (3/2) l)$ 
	form a basis of the lattice; 
\item representatives $v_{0}$ and $v_{1}$ of the vertices are mapped to 
	$\Phi^{(0)}(v_{0}) = (0,0)$ and $\Phi^{(0)}(v_{1}) = ((\sqrt{3}/2) l, (1/2) l)$; and 
\item The non-loop edges $e_{1}, e_{2}$, and $e_{3}$ originating from $v_{0}$ are mapped to 
	$\Phi^{(0)}(e_{1}) = ((\sqrt{3}/2) l, (1/2) l), 
	\Phi^{(0)}(e_{2}) = (-(\sqrt{3}/2) l, (1/2) l)$, 
	and $\Phi^{(0)}(e_{3}) = (0,-l)$. 
\end{enumerate}

Then, the volume per period is $\mathcal{V} = (3\sqrt{3}/2) l^{2}$. 
The energy per period is $\mathcal{E}^{(0)} = \mathcal{E}(X^{(0)},\Phi^{(0)}) = 3w_{1}l^{2}$. 
Young's modulus is $E = 4\mathcal{E}_{0} / \mathcal{V} = (8\sqrt{3}/3)w_{1}$. 
The tension tensors around the vertex $v_{0}$ and $v_{1}$ are 
$\mathcal{T}(v_{0}) = \mathcal{T}(v_{1}) = 
\begin{pmatrix}
(3/2)w_{1}l^{2} & 0 \\
0 & (3/2)w_{1}l^{2}
\end{pmatrix}$.

\twofig{width=6cm}{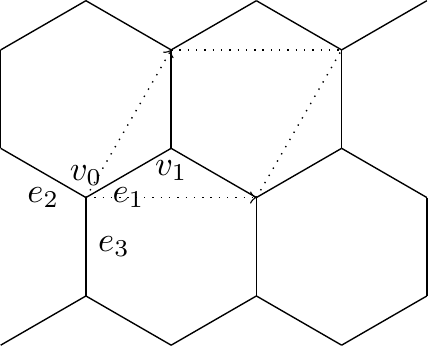}{The hexagonal lattice}{width=6cm}{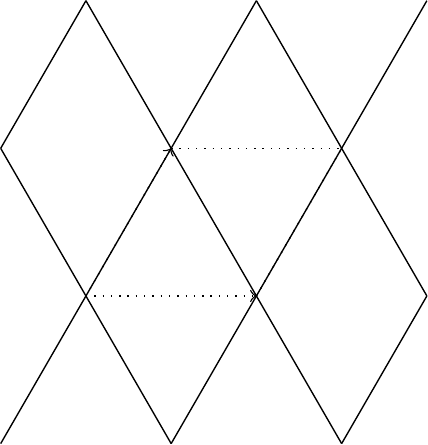}{The graph $X^{(1)}$}

\twofig{width=6cm}{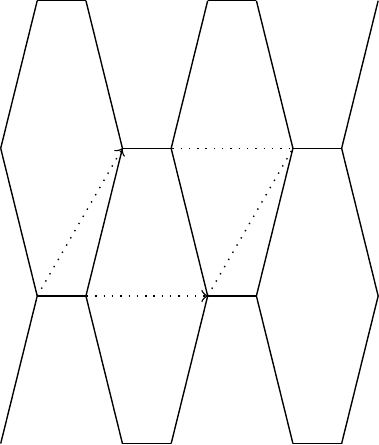}{The graph $X^{(2)}$}{width=6cm}{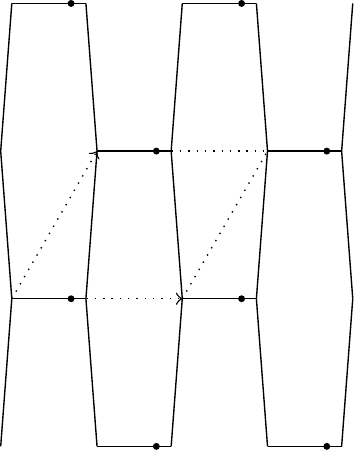}{The graph $X^{(3)}$}

\subsubsection{The case $\theta = 0$}

Consider the slow deformation by $A_{t} = A(\lambda^{t})$ for $\theta = 0$. 
Then, 
\[
\mathcal{E}^{(0)}(\lambda) = \frac{3}{2}w_{1}l^{2}(\lambda^{2}+\lambda^{-2}), \quad 
\sigma_{\mathrm{eng}}^{(0)}(\lambda) = \frac{2\sqrt{3}}{3}w_{1}(\lambda-\lambda^{-3}). 
\]

%\epsilon_{0}^{(0)} = 0 

Suppose that 
$1 < \delta^{-1}l < \sqrt{2K_{2w_{0}+3w_{1}} / 3w_{1}l^{2}}$. 
When $\lambda^{t} = \delta^{-1}l$, the edge $e_{3}$ contracts to a loop, 
where no splittings occurred earlier. 
We obtain a periodic graph $X^{(1)}$ 
as shown in Figure~\ref{fig:hex11.pdf}. 
The edges of $X^{(1)}$ consist of loops with weight $2w_{0}+w_{1}$ and the 1-skeleton of the square tiling. 
Then, 
\[
\mathcal{E}^{(1)}(\lambda) = \frac{3}{2}w_{1}l^{2}(\lambda^{2}+3\lambda^{-2}), \quad 
\sigma_{\mathrm{eng}}^{(1)}(\lambda) = \frac{2\sqrt{3}}{3}w_{1}(\lambda-3\lambda^{-3}), 
\]
\[
\epsilon_{0}^{(1)} = \sqrt[4]{3}-1 \approx 0.316, \quad 
R^{(1)} = -1. 
\]

Furthermore, when $\lambda^{t} = \sqrt{K_{4w_{0}+6w_{1}} / 3w_{1}l^{2}}$, 
the vertices of $X^{(1)}$ split. 
We obtain a periodic graph $X^{(2)}$
as shown in Figure~\ref{fig:hex12.pdf}. 
The weights of each loop and new non-loop edge are, respectively, 
$(1/2)w_{0}+(1/4)w_{1}$ and $w_{0}+(1/2)w_{1}$. 
Then, 
\begin{align*}
\mathcal{E}^{(2)}(\lambda) &= \mathcal{E}^{(1)}(\lambda) - \frac{3w_{1}^{2}l^{2}}{w_{0}+(5/2)w_{1}}\lambda^{2} \\ 
&= \frac{3}{2}w_{1}l^{2} \left( \frac{2w_{0}+w_{1}}{2w_{0}+5w_{1}}\lambda^{2}+3\lambda^{-2} \right), \\
\sigma_{\mathrm{eng}}^{(2)}(\lambda) 
&= \frac{2\sqrt{3}}{3}w_{1} \left( \frac{2w_{0}+w_{1}}{2w_{0}+5w_{1}}\lambda-3\lambda^{-3} \right), 
\end{align*}
\[
\epsilon_{0}^{(2)} = \sqrt[4]{\frac{3(2w_{0}+5w_{1})}{2w_{0}+w_{1}}}-1, \quad 
R^{(2)} = -\frac{2w_{0}+3w_{1}}{2w_{0}+5w_{1}}. 
\]
(The first equality is also obtained by Theorem~\ref{thm:loss}.) 
It holds that 
$\sqrt[4]{3}-1 < \epsilon_{0}^{(2)} \leq \sqrt[4]{15}-1 \approx 0.967$. 
If $w_{0} = w_{1}$, 
then $\epsilon_{0}^{(2)} = \sqrt[4]{7}-1 \approx 0.626$. 
We remark that $\epsilon_{0}^{(2)}$ decreases 
as $w_{0} / w_{1}$ increases.

Furthermore, when $\lambda^{t} = \sqrt{\dfrac{2w_{0}+5w_{1}}{2w_{0}+w_{1}} \dfrac{2K_{2w_{0}+3w_{1}}}{3w_{1}l^{2}}}$, 
vertices of $X^{(2)}$ split. 
For genericity, 
we suppose that a single vertex per period splits. 
(We may assume that a representative thereof is at the origin.) 
We obtain a periodic graph $X^{(3)}$ 
as shown in Figure~\ref{fig:hex13.pdf}. 
Then, 
\begin{align*}
\mathcal{E}^{(3)}(\lambda) 
&= \mathcal{E}^{(2)}(\lambda) - 
\frac{2w_{0}+w_{1}}{2w_{0}+5w_{1}} \frac{24w_{1}^{2}l^{2}}{2w_{0}+21w_{1}}\lambda^{2} \\ 
&= \frac{3}{2}w_{1}l^{2} \left( \frac{2w_{0}+w_{1}}{2w_{0}+21w_{1}}\lambda^{2}+3\lambda^{-2} \right), \\
\sigma_{\mathrm{eng}}^{(3)}(\lambda) 
&= \frac{2\sqrt{3}}{3}w_{1} \left( \frac{2w_{0}+w_{1}}{2w_{0}+21w_{1}}\lambda-3\lambda^{-3} \right), 
\end{align*}
\[
\epsilon_{0}^{(3)} = \sqrt[4]{\frac{3(2w_{0}+21w_{1})}{2w_{0}+w_{1}}}-1, \quad 
R^{(3)} = -\frac{2w_{0}+11w_{1}}{2w_{0}+21w_{1}}. 
\]
It holds that 
$\sqrt[4]{3}-1 \leq \epsilon_{0}^{(3)} \leq \sqrt[4]{63}-1 \approx 1.817$. 
If $w_{0} = w_{1}$, 
then $\epsilon_{0}^{(3)} = \sqrt[4]{23}-1 \approx 1.189$.

\twofig{width=6.5cm}{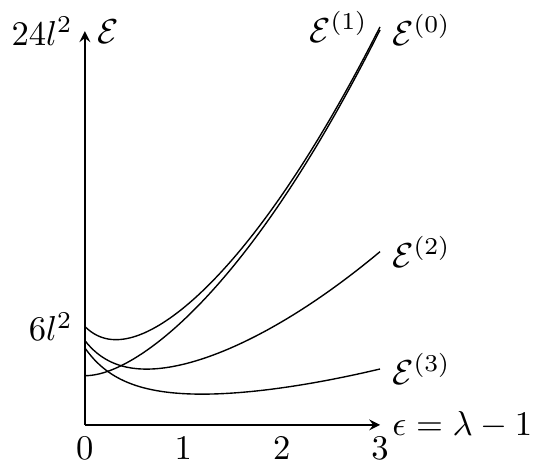}{Energies for deformation in the case $\theta = 0$}{width=6cm}{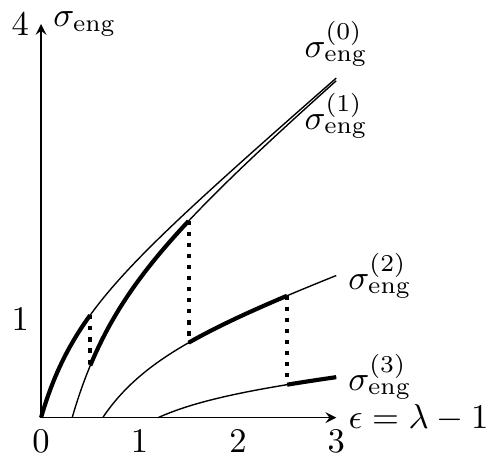}{Stress--strain curve in the case $\theta = 0$}

Similarly, 
more splittings of vertices may occur. 
Subsequently, the length of a path of edges increases. 
For $w_{0} = w_{1} = 1$, 
the energies are shown in Figure~\ref{fig:hex1energy.pdf}, 
and the stress--strain curve is drawn as the thick discontinuous curve 
in Figure~\ref{fig:hex1stress.pdf}. 
Note that if we take a larger period, 
the occurrences of local moves change by genericity. 
We must also consider a contraction when $\lambda^{t} = 3\delta^{-1}l$. 
With this consideration, the stress--strain curve also changes. 
It would be desirable for 
these stress--strain curves to converge to a continuous curve 
as the periods expand.

\subsubsection{The case $\theta = \pi/6$}

Consider the slow deformation by $A_{t} = A(\lambda^{t})$ for $\theta = \pi/6$. 
Then, only splittings occur. 
We state only a result: 
\[
\mathcal{E}^{(0)}(\lambda) = \frac{3}{2}w_{1}l^{2}(\lambda^{2}+\lambda^{-2}), \quad 
\sigma_{\mathrm{eng}}^{(0)}(\lambda) = \frac{2\sqrt{3}}{3}w_{1}(\lambda-\lambda^{-3}) 
\quad \text{(same as above)}, 
\]
\begin{align*}
\mathcal{E}^{(1)}(\lambda) &= \frac{3}{2}w_{1}l^{2} \left( \frac{3w_{0}}{3w_{0}+4w_{1}}\lambda^{2}+\lambda^{-2} \right), & 
\sigma_{\mathrm{eng}}^{(1)}(\lambda) &= \frac{2\sqrt{3}}{3}w_{1} \left( \frac{3w_{0}}{3w_{0}+4w_{1}}\lambda-\lambda^{-3} \right), \\
\epsilon_{0}^{(1)} &= \sqrt[4]{\frac{3w_{0}+4w_{1}}{3w_{0}}}-1, &  
R^{(1)} &= \frac{2w_{1}}{3w_{0}+4w_{1}}, \\
\mathcal{E}^{(2)}(\lambda) &= \frac{3}{2}w_{1}l^{2} \left( \frac{3w_{0}}{3w_{0}+8w_{1}}\lambda^{2}+\lambda^{-2} \right), & 
\sigma_{\mathrm{eng}}^{(2)}(\lambda) &= \frac{2\sqrt{3}}{3}w_{1} \left( \frac{3w_{0}}{3w_{0}+8w_{1}}\lambda-\lambda^{-3} \right), \\
\epsilon_{0}^{(2)} &= \sqrt[4]{\frac{3w_{0}+8w_{1}}{3w_{0}}}-1, & 
R^{(2)} &= \frac{4w_{1}}{3w_{0}+8w_{1}}, \\
\mathcal{E}^{(3)}(\lambda) &= \frac{3}{2}w_{1}l^{2} \left( \frac{w_{0}}{w_{0}+8w_{1}}\lambda^{2}+\lambda^{-2} \right), & 
\sigma_{\mathrm{eng}}^{(3)}(\lambda) &= \frac{2\sqrt{3}}{3}w_{1} \left( \frac{w_{0}}{w_{0}+8w_{1}}\lambda-\lambda^{-3} \right), \\
\epsilon_{0}^{(3)} &= \sqrt[4]{\frac{w_{0}+8w_{1}}{w_{0}}}-1, & 
R^{(3)} &= \frac{4w_{1}}{w_{0}+8w_{1}}, 
\end{align*}
and so on. 
For $w_{0} = w_{1} = 1$, 
the energies are shown in Figure~\ref{fig:hex2energy.pdf}, 
and the stress--strain curve is shown in Figure~\ref{fig:hex2stress.pdf}.

\twofig{width=6.5cm}{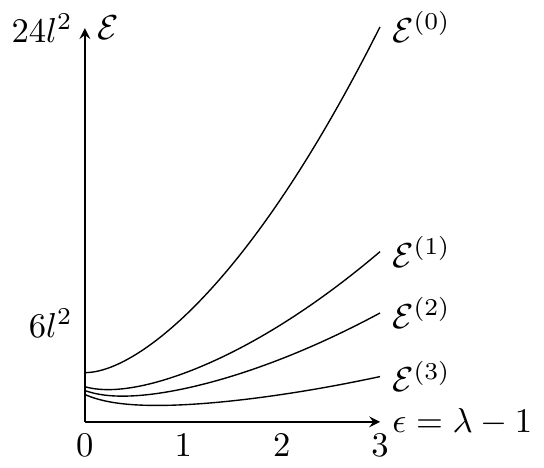}{Energies for deformation in the case $\theta = \pi/6$}{width=6cm}{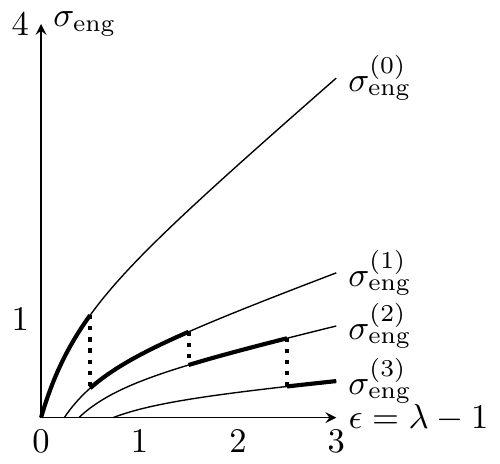}{Stress--strain curve in the case $\theta = \pi/6$}

\section{Discussion and perspectives} 
\label{section:discussion}

We discuss some open issues 
and propose questions for further research. 
\begin{enumerate}
\item 
The definition of the tension tensor fits with our purely mathematical interest, 
and we can apply it to topological and discrete geometric properties of graphs and nets. 
%On the other hand, we can consider another function of the length to define energy and tension tensor. 
%Under this assumption, a linear transformation of a harmonic realization is no longer harmonic, but what can we say? 
To define an energy, we have used $E(l) = l^{2}$ 
as an energy of an edge of length $l$. 
What happens when we take another energy function $E(l)$? 
A linear transformation of a harmonic realization is no longer harmonic, 
but a harmonic realization is still unique if $E^{\prime}(l) > 0$ and $E^{\prime \prime}(l) > 0$. 

\item 
In Section~\ref{section:move}, 
we have arbitrarily given the threshold values $\delta$ and $K_{d}$ 
in the conditions of a contraction and a splitting. 
What are the values $\delta$ and $K_{d}$ for an actual TPE? 
How does the value $K_{d}$ depend on the degree $d$? 

\item 
The validity of our model for deformation requires 
the assumption that infinite repetition of local moves does not occur. 
However, mutually inverse splittings and contractions may continue alternatingly 
in some cases. 
Theorem~\ref{thm:compatibility} gives a sufficient condition that 
such alternating repetition does not occur. 
Is another kind of infinite repetition possible? 
If it is possible, what condition is sufficient to avoid such repetition? 

\item 
In Section~\ref{section:ex}, 
we have given simple examples for two-dimensional nets. 
What about three-dimensional nets? 

\item 
Do local moves occur in unloading? 
If not, then it reproduces the Mullins effect~\cite{dfg09review}: 
the stress--strain curve in reloading 
coincides with that in the unloading 
until the maximal strain of the prior loading. 
Figures~\ref{fig:hex1stress.pdf} and \ref{fig:hex2stress.pdf} illustrate this behaviour. 
Even if local moves occur in unloading, 
the inverse local moves in reloading may induce the Mullins effect. 

\item 
We have worked in continuous weights of edges, 
which is useful for mathematical arguments. 
Of course, the weights corresponding to actual polymers are integers. 
However, it would not be meaningful only to restrict ourselves to integer weights. 
Our presented settings preserve periodicity. 
As a result, stress--strain curves are not continuous. 
To obtain a continuum limit of such discrete matters, 
stochastic formulation of local moves will be effective. 
The stochastic formulation in integer weights 
may induce continuous stress--strain curves 
and a non-linear constitutive equation that reflects elastoplasticity. 
We expect that this is also appropriate to describe fracture of materials. 
\end{enumerate}

\section*{Funding} 
This study is supported by JST CREST Grant Number JPMJCR17J4. 
The first author is also supported by the World Premier International Research Center Initiative (WPI), MEXT, Japan. 
The second author is also supported by JSPS KAKENHI Grant Number 19K14530. 

\section*{Acknowledgements} 
The authors are grateful to 
Yoshifumi Amamoto, Tetsuo Deguchi, Ken Kojio, Motoko Kotani, Hiroshi Morita, Ken Nakajima, Genki Omori, and Koya Shimokawa
for their helpful discussions. 
The authors 
would like to thank Editage (www.editage.com) for English language editing, 
and the anonymous reviewers for their suggestions and comments.

\bibliographystyle{siam}
\bibliography{ref-crest1}

\end{document}